\documentclass[10pt,a4paper,reqno]{article} 
\usepackage{mysty}

\usepackage[thinlines]{easytable}

\usepackage{fullpage}
\usepackage{multirow}
\usepackage{graphicx}

\SetLabelAlign{parright}{\parbox[t]{\labelwidth}{\raggedleft{#1}}}
\setlist[description]{style=multiline,topsep=4pt,align=parright}

\graphicspath{{./fig/}}

\makeatletter
\let\reftagform@=\tagform@
\def\tagform@#1{\maketag@@@{(\ignorespaces\textcolor{black}{#1}\unskip\@@italiccorr)}}
\newcommand{\iref}[1]{\textup{\reftagform@{\tcr{\ref{#1}}}}}
\makeatother

\usepackage[noindentafter]{titlesec}
\usepackage{bold-extra}

\titlespacing\section{0pt}{11pt plus 4pt minus 2pt}{6pt plus 2pt minus 2pt}
\titlespacing\subsection{0pt}{10pt plus 4pt minus 2pt}{4pt plus 2pt minus 2pt}
\titlespacing\subsubsection{0pt}{8pt plus 4pt minus 2pt}{4pt plus 2pt minus 2pt}
\titlespacing\paragraph{0pt}{6pt plus 4pt minus 2pt}{6pt plus 2pt minus 2pt}

\newcommand{\tnabla}{\widetilde{\nabla}}
\newcommand{\tnablasaga}{\widetilde{\nabla}^{\textnormal{\tiny SAGA}}}

\newcommand{\tnablasarah}{\widetilde{\nabla}^{\textnormal{\tiny SARAH}}}

\begin{document}

	\setlength{\abovedisplayskip}{4.5pt}
	\setlength{\belowdisplayskip}{4pt}

\title{A Stochastic Alternating Direction Method of Multipliers for Non-smooth and Non-convex Optimization}
\author{Fengmiao Bian\thanks{School of Mathematical Sciences, Shanghai Jiao Tong University, Shanghai China. E-mail: bianfm17@sjtu.edu.cn} \and
		Jingwei Liang\thanks{School of Mathematical Sciences, Queen Mary University of London, London UK. E-mail: jl993@cam.ac.uk.} \and
		Xiaoqun Zhang\thanks{School of Mathematical Sciences and Institute of Natural Sciences, Shanghai Jiao Tong University, Shanghai China. E-mail: xqzhang@sjtu.edu.cn.} 
		}
\date{}
\maketitle

\begin{abstract}
Alternating direction method of multipliers (ADMM) is a popular first-order method owing to its simplicity and efficiency. However, similar to other proximal splitting methods, the performance of ADMM degrades significantly when the scale of the optimization problems to solve becomes large. 
In this paper, we consider combining ADMM with a class of stochastic gradient with variance reduction for solving large-scale non-convex and non-smooth optimization problems. 
Global convergence of the generated sequence is established under the extra additional assumption that the object function satisfies Kurdyka-\L ojasiewicz (KL) property. 
Numerical experiments on graph-guided fused Lasso and computed tomography are presented to demonstrate the performance of the proposed methods.
\end{abstract}

\begin{keywords}
Non-convex optimization $\cdot$ 
stochastic ADMM $\cdot$ variance reduction stochastic gradient 
\end{keywords}

\begin{AMS}
{90C26 $\cdot$ 90C30 $\cdot$ 90C90 $\cdot$ 15A83 $\cdot$ 65K05}
\end{AMS}


\section{Introduction}

Driven by the problems arising from diverse fields including signal/image processing, compressed sensing, inverse problems, computer vision and many others, the past decades have witnessed a tremendous success of non-smooth optimization and first-order proximal splitting algorithms \cite{chambolle2016introduction}. 
Nowadays, with the advances of data acquisition tools and mathematical modeling, the problems to handle are becoming increasingly complex which imposes challenges, \eg large dimension and non-smoothness, to design efficient numerical algorithms. 
In the fields of machine learning and related areas, stochastic optimization methods are widely adopted due its simplicity and fast convergence. Over the past few years, in traditional areas such as imaging science,  stochastic methods are becoming popular due to the aforementioned challenges. In this paper, we follow this path and consider a stochastic version of the popular alternating direction method of multipliers (ADMM) \cite{GM, GM1}.


In this paper, we are interested in solving the following composite optimization problem
\begin{equation}\label{model}
\min_{x \in \mathbb{R}^d} H(x) + F(Ax) ,
\end{equation}
where $H = \frac{1}{n}\sum_{i=1}^n H_i(x) $ has finite sum structure, $A: \mathbb{R}^d \to \mathbb{R}^m$ is a linear mapping and $H, F$ are proper lower semi-continuous functions. Throughout this paper, no convexity is imposed to either $H$ or $F$. 
%
In practice, numerous problems can be formulated in to the form of \eqref{model}. 
For example, in computed tomography (CT), the reconstruction task can be modeled as
\begin{equation}\label{CT}
\min_{x \in \bbR^d} \sfrac{1}{n} \sum_{i=1}^n ( \mathcal{R}_i x - b_i )^2 + F(Ax),
\end{equation}
where $x$ is the image to be reconstructed, $n$ is the number of projection, $\mathcal{R}$ is Radon transform, $b_i$ is the $i_{th}$ observation. 
In order to ensure the quality of the reconstructed image, a large amount of  projection is required, which means a large value of $n$. 
%
%
Another example is classification/regression, where one needs to solve the following problem
\begin{equation}\nonumber
\min_{x\in\bbR^d} \sfrac{1}{n} \sum_{i=1}^n H(a_i^T x, y_i) + F(Ax),
\end{equation}
where for each $i=1,...,n$, $(a_i, y_i) \in \bbR^{n} \times \bbR$ is given training data and its corresponding label; $H$ is the loss function, such as square loss or logistic loss. In many cases, one needs to use a large scale of training data which results in huge value of $n$. 
In both cases, $F(Ax)$ is called regularization term, where $A$ is a properly chosen linear transformation such as discrete gradient operator or wavelet transform, and $F$ is a low-complexity promoting functions such as $L_0$ or $L_1$-norms for sparsity. We refer to \cite{CNZ,ZDL,WWT,KSX,BD} and the references therein for more detailed discussions. 


\subsection{Alternating direction method of multipliers}

In the literature, one popular method to solve \eqref{model} is alternating direction method of multipliers (ADMM). ADMM was developed around 1970s, where \cite{GM, GM1} shaped the original form of the algorithm, deeper theoretical understanding of ADMM can be found in \cite{eckstein1992douglas} and the nowadays popularity of ADMM partially owes to \cite{boyd2011distributed}. 
To apply ADMM to solve \eqref{model}, one needs to add an auxiliary variable $z$ which leads to the following constrained problem
\begin{equation}\label{constraint}
\min_{x \in \mathbb{R}^d, u \in \mathbb{R}^m} H(x) + F(z) \quad \textrm{such~that} \quad z = Ax .
\end{equation}
The augmented Lagrangian formulation associated to \eqref{constraint} reads
\[
\mathcal{L}_{\beta}(x, z, u) = F(z) + H(x) + \langle u, Ax - z \rangle + \sfrac{\beta}{2} \| Ax - z \|^2,
\]
where $u$ is the Lagrange multiplier and $\beta > 0$ is properly chosen \cite{GM,GM1}. 
In this paper, we adopt a linearized ADMM proposed in \cite{YY, ZBO, OCLP, BN,LSG}, which takes the following form of iteration
\begin{equation}\label{lADMM}
\begin{aligned}
z_{t+1} &\in  \arg\min_{z}   F(z) + \langle u_t, Ax_t - z \rangle + \sfrac{\beta}{2}\| Ax_t - z \|^2  ,\\
x_{t+1} &=  x_t - \sfrac{1}{\tau} \Pa{ \nabla  H(x_t) + A^{T} u_t + \beta A^{T} (Ax_t - z_{t+1}) } ,   \\
u_{t+1} &= u_{t} +\sigma \beta (Ax_{t+1} - z_{t+1}) ,
\end{aligned}
\end{equation}
where $\tau, \sigma > 0$ are step-sizes. The difference between \eqref{lADMM} and the standard ADMM lies in the update of $x_{t+1}$, for standard ADMM one needs to solve a proximal minimization of $H(x)$, while \eqref{lADMM} takes the gradient descent for $H$. 

Convergence guarantees of \eqref{lADMM} in the convex case can be found in \cite{YY, ZBO,OCLP}, while for the non-convex case results are available in \cite{BN, LSG}. 
For the case of $H$ being a finite sum, computing the full gradient $\nabla H$ can be very time consuming which damps the efficiency of the method. To circumvent such difficulty, one can replace $\nabla H$ with its stochastic approximation, which results in a stochastic version of \eqref{lADMM} and is the core of our contributions in this work.

\subsection{Contributions}

Motivated by the work of \cite{DTLDS}, in this paper we propose a stochastic version of ADMM \eqref{lADMM} for solving non-smooth and non-convex objective \eqref{model}. 
The key of our algorithm is replacing the full gradient computation $\nabla H$ in \eqref{lADMM} with its stochastic approximations $\tnabla H$, and the resulted algorithm is described below in Algorithm~\ref{alg:SADMM}.

\begin{center}
\begin{minipage}{0.95\linewidth}
	
\begin{algorithm}[H]
\caption{A Stochastic Alternating Direction Method of Multipliers (SADMM)}
\label{alg:SADMM}
\begin{algorithmic}
\STATE{{\bf{Step 0.}}  Input $T$ and $\beta,~\tau,~\sigma > 0$. Initialize $x_0$ and $u_0$.}

\STATE{{\bf{Step 1.}}  For $t=0, 1, \dots, T-1$ do
\begin{subequations}\label{alg1}
\begin{align}
z_{t+1} &\in  \arg\min_{z}   F(z) + \langle u_t, Ax_t - z \rangle + \sfrac{\beta}{2}\| Ax_t - z \|^2  , \label{sub1}\\
x_{t+1} &=  x_t - \sfrac{1}{\tau} \Pa{ \tnabla  H(x_t) + A^{T} u_t + \beta A^{T} (Ax_t - z_{t+1}) } ,  \label{sub2} \\
u_{t+1} &= u_{t} +\sigma \beta (Ax_{t+1} - z_{t+1}). \label{sub3}
\end{align}
\end{subequations}
}
\STATE{{\bf{Step 2.}}  Iterate $x$ and $z$ chosen uniformly random from $\{x_t, z_t\}_{t=1}^T$.}

\end{algorithmic}
\end{algorithm}

\end{minipage}
\end{center}

Various stochastic gradient approximations developed in the literature can be used by $\tnabla H$, for instance the simplest choice is the stochastic gradient \cite{RM}. Since $H$ has finite sum structure, variance reduced stochastic gradient approximations can also be applied, such as SAGA \cite{DBL} and SVRG \cite{JZ} which are unbiased, and SARAH \cite{NLST} which is biased. 
Overall, our contributions in this paper consist the following
\begin{itemize}
\item[1.] By combing ADMM with stochastic gradient approximations, we propose a stochastic ADMM algorithm for non-convex and non-smooth optimization. Our algorithm is very general in the sense that it can incorporate with most existing stochastic gradient approximation in the literature. 

\item[2.] When the stochastic gradient approximation $\tnabla H$ is variance reduced (see Definition \ref{var-redu}), a decent property is obtained for a specifically designed stability function which yields the convergence of the objective function. If moreover the objective function has \KL property (see Definition \ref{KL1}), convergence of the generated sequences is also established. 

%
%

\item[3.] Numerically,  we test performance of SADMM using several well studied stochastic gradient approximations (\eg SGD, SAGA, SVRG and SARAH) on problems including graph-guided fused Lasso, wavelet frame based 2D CT reconstruction and TV-L0 3D CT reconstruction. We also compare with some existing algorithms in the literature, and our result indicates SADMM achieves the best performance using SARAH gradient approximation.
\end{itemize}

\subsection{Related work}


Over the past years, extension of ADMM to the case of non-convex optimization and stochastic setting are widely studied. Along the direction of non-convex optimization, recently some theoretical foundations are established. 
For instance in \cite{LPADMM}, the authors studied the convergence of ADMM for a class of non-convex optimization problems. In \cite{WCX, WCX1}, convergence of a non-convex Bregman ADMM was studied. For non-convex separable objective functions with applications to consensus problem, the authors of \cite{HLR} studied the properties of ADMM. Other related work can be found in for instance \cite{WYZ,BCN} under different settings.  
It should be noted that all these works are studied under the deterministic setting, that is no randomness is involved in these results. 

For stochastic versions of ADMM, the first attempt can be found in \cite{WB}, where the authors proposed an online ADMM for large scale optimization. 
In recent years, the success of stochastic proximal gradient algorithms, such as SAG/SAGA \cite{SRB, DBJ}, SDCA \cite{SZ}, SVRG \cite{JZ, KLRT} and SARAH \cite{NLST}, greatly boost the studies of stochastic ADMM. 
For example, in \cite{ZK} the authors combined ADMM with SAG gradient approximation, which was further combined with Nesterov acceleration \cite{NY} in \cite{AS}. 
Stochastic ADMM with SDCA was studied in \cite{TS} for a wide range of regularized learning problems, and linear rate of convergence was obtained when the objective function obeys some strong convexity and smoothness property. 
The combinations of ADMM with SVRG can be found in \cite{ZK1,LSC} under different settings. 
We also refer to \cite{ZLZ} for related developments of stochastic ADMM.  
{In \cite{HC} the authors studied the combination of ADMM and three different gradient approximations: SGD, SVRG and SAGA,  for solving non-convex non-smooth  problems.  For each gradient approximation, the authors proved that $(x_{t^*}, z_{t^*})$ is an $\epsilon$-stationary point of the objective function, where $t^*= \arg\min_{2 \leq t \leq T+1} \theta_t$  with $\theta_t$ being a variable associated the gradient estimator, while the authors showed that the mini-batch SGD/SVRG/SAGA-ADMM  has a convergence rate of $\mathcal{O}(\sfrac{1}{T})$. Similar approach is used in \cite{HCH} for the convergence analysis of a faster non-convex stochastic ADMM, which using a stochastic path-integrated differential estimator.}


In comparison, our work is motivated by \cite{DTLDS} where the authors proposed a generic stochastic version proximal alternating linearized minimization (PALM) algorithm \cite{BST}, for which various variance-reduced gradient approximations are allowed. 
We extend the result of \cite{DTLDS} to the case of ADMM, hence our proposed SADMM is very general in the sense that Algorithm \ref{alg:SADMM} works with any stochastic gradient approximations which are variance reduced.

\paragraph{Paper organisation} 
The rest of the paper is organized as following. In Section \ref{sec:notation}, we collect some useful preliminary results which are essential to our analysis. The main theoretical analysis of our proposed algorithm can be found in Section \ref{sec:convergence}, followed by numerical result in Section \ref{sec:examples}. Proofs of main theorems are differed in the appendix.

\section{Preliminaries}
\label{sec:notation}

In this section, we give some notation  and definitions used in our paper. For theoretical analysis, we first recall the definition of variance reduction gradient estimators which is taken from \cite{DTLDS}.

\begin{definition}[{\cite[Definition 2.1]{DTLDS}}]\label{var-redu}
A gradient estimator $\tnabla $ is called variance-reduced if there exist constants $V_1,~V_2,~V_{\Upsilon} \geq 0$ and $\rho \in (0, 1]$ such that
\begin{itemize}
\item[1.] $(${\bf MSE Bound}$)$ There exists a sequence having random variables $\{ \Upsilon_t \}_{t \geq 1}$ of the form $\Upsilon_t = \sum_{i=1}^s \| v_t^i \|^2$ for some random vectors $v_t^i$ such that
\begin{equation}\nonumber
\mathbb{E}_t \| \tnabla  H(x_t) - \nabla H(x_t) \|^2 \leq \Upsilon_t + V_1 (\mathbb{E}_t \| x_{t+1} - x_t \|^2 + \| x_t - x_{t-1} \|^2 ),
\end{equation}
and, with $\Gamma_t = \sum_{i=1}^s \| v_t^i \|,$
\begin{equation}\label{var-mse1}
\mathbb{E}_t \| \tnabla  H(x_t) - \nabla H(x_t) \| \leq \Gamma_t + V_2 (\mathbb{E}_t \| x_{t+1} - x_t \| + \| x_t - x_{t-1} \|).
\end{equation}
\item[2.] $(${\bf Geometric Decay}$)$ The sequence $\{ \Upsilon_t \}_{t\geq 1}$ decays geometrically:
\begin{equation}\label{var-geo}
\mathbb{E}_t \Upsilon_{t+1} \leq (1 - \rho) \Upsilon_t + V_{\Upsilon} \Pa{ \mathbb{E}_t \| x_{t+1} - x_t \|^2 + \| x_t - x_{t-1} \|^2 }.
\end{equation}
\item[3.] $(${\bf Convergence of Estimator}$)$ For all sequences $\{ x_t \}_{t = 0}^{\infty}$ satisfying $\lim_{t \to \infty} \mathbb{E} \| x_t - x_{t-1} \|^2$ $ \to 0,$ it follows that $\mathbb{E} \Upsilon_t \to 0$ and $\mathbb{E} \Gamma_t \to 0$.
\end{itemize}
\end{definition}

As remarked in \cite{DTLDS}, most existing variance reduced gradient estimators in the literature satisfy the above definition, in this work we mainly consider SAGA \cite{DBL} and SARAH \cite{NLST} which are widely used. 
For these two estimators, we provide their definitions below and refer to Appendix \ref{app1} for their properties.

\begin{definition}[SAGA \cite{DBL}]\label{SAGA}
The SAGA gradient approximation $\tnablasaga H(x)$ is defined as follows:
\begin{equation}\nonumber
\tnablasaga H(x_t) 
= \sfrac{1}{b} \bPa{ \sum\nolimits_{j \in J_t} \nabla H_j (x_t) - \nabla H_j (\varphi_t^j) } + \sfrac{1}{n} \sum\nolimits_{i = 1}^n \nabla H_i (\varphi_t^i),
\end{equation}
where $J_t$ is mini-batches containing $b$ indices. The variables $\varphi_t^i$ follow the update rules $\varphi_{t+1}^i = x_t$ if $i \in J_t$ and $\varphi_{t+1}^i = \varphi_t^i$ otherwise. 
\end{definition}

\begin{definition}[SARAH \cite{NLST}]\label{sarah}
The SARAH estimator reads for $t=0$ as 
\begin{equation}\nonumber
\tnablasarah H(x_0) = \nabla H(x_0).
\end{equation}
For $t = 1, 2, \dots $, define random variables $p_t \in \{ 0, 1 \}$ with $P(p_t = 0) = \frac{1}{p}$ and $P(p_t = 1) = 1 - \frac{1}{p}$, where $p \in (0, \infty)$ is a fixed chosen parameter. Let $J_t$ be a random subset uniformly drawn from $\{ 1, \dots, n \}$ of fixed batch size $b$. Then for $t = 1, 2, \dots$ the SARAH gradient approximation reads as 
\begin{equation}\nonumber
\tnablasarah H(x_t) 
= 
\left\{
\begin{aligned}
\nabla H(x_t) &: \textrm{if}~p_t = 0,\\
\sfrac{1}{b} \Pa{ \ssum_{j \in J_t} \nabla H_j (x_t) - \nabla H_j (x_{t-1}) } + \tnablasarah H(x_{t-1}) &: \textrm{if}~p_t = 1.
\end{aligned}
\right.
\end{equation}
\end{definition}


Below we include supermartingale convergence result, and refer to \cite{D} and \cite{RS} for more details.

\begin{lemma}\label{lemma1}
Suppose $x_1, \dots, x_t$ are independent random variables satisfying $\mathbb{E}_t x_i = 0$ for all $i$. Then
\begin{equation}\nonumber
\mathbb{E}_t \| x_1 + \cdots + x_t \|^2 = \mathbb{E}_t \big[ \| x_1 \|^2 + \cdots + \| x_t \|^2 \big].
\end{equation}
\end{lemma}

\begin{lemma}[Supermartingale Convergence Theorem]\label{super}
Let $\mathbb{E}_t$ denote the expectation conditional on the first $t$ iterations of Algorithm \ref{alg:SADMM}. Let $\{ X_t \}_{t=0}^{\infty}$ and $\{ Y_t \}_{t=0}^{\infty}$ be sequences of bounded non-negative random variables such that $X_t$ and $Y_t$ depend only on the first $t$ iterations of Algorithm \ref{alg:SADMM}. If 
\begin{equation}\nonumber
\mathbb{E}_t X_{t+1} + Y_t \leq X_t,
\end{equation}
then $\sum_{t=0}^{\infty} Y_t< \infty~a.s.$ and $X_t$ converges $a.s.$.
\end{lemma}

As we are considering non-convex problem, to deliver convergence analysis, we need Kurdyka-\L ojasiewicz (KL) property which is nowadays widely used in non-convex optimization. The definition of KL inequality is provided below and we refer to \cite{ABRS, ABS, BDL} and the references therein for more detailed accountant. 
For $\epsilon_1, \epsilon_2$ satisfying $\ninf < \epsilon_1 < \epsilon_2 < \pinf$, define the set $[ \epsilon_1< F < \epsilon_2 ] \eqdef \ba{ x\in \bbR^{m_1} : \epsilon_1 < F(x) < \epsilon_2}$.

\begin{definition}[KL inequality]\label{KL1}
A function $F: \mathbb{R}^{n} \to \mathbb{R} \cup \{+\infty\}$ has the Kurdyka-\L ojasiewicz property at $x^{*} \in$ dom $\partial F$ if there exist $\eta \in (0, \infty]$, a neighborhood $U$ of $x^{*}$, and a continuous concave function $\varphi : [0, \eta) \to \mathbb{R}_{+}$ such that:

\begin{itemize}
\item [(i)]  $\varphi(0) = 0,~\varphi \in C^{1}((0, \eta))$, and  $\varphi^{'}(s) > 0$ for all $s \in (0, \eta)$;

\item[(ii)] for all $x \in U \cap [F(x^{*}) < F < F(x^{*})+ \eta]$ the Kurdyka-\L ojasiewicz inequality holds, i.e.,
$$ 
\varphi^{'}(F(x) - F(x^{*}))\dist(0, \partial F(x)) \geq 1.
$$
\end{itemize}
If $F$ satisfies the Kurdyka-\L ojasiewicz property  at each point of dom $\partial F$, then it is called a KL function.
\end{definition}

Roughly speaking, KL functions become sharp up to reparameterization via $\varphi$, called a desingularizing function for $R$. Typical KL functions are the class of semi-algebraic functions, see \cite{BDL, BDLM}. For instance, the $\ell_0$ pseudo-norm and the rank function are KL.

\section{Convergence analysis}
\label{sec:convergence}


In this section, we provide convergence analysis for Algorithm \ref{alg:SADMM} for solving \eqref{model}. To this end, we need some basic assumptions which are listed below.

\begin{assumption}\label{assum}
For problem \eqref{model}, we suppose that
\begin{enumerate}[leftmargin=4em, label= {\textrm{{\bf A.\arabic{*})}}}, ref= \textrm{\bf A.\arabic{*}}]
\item Functions $F :  \bbR^m \to \bbR\cap\{+\infty\}$ and $H: \bbR^d \to \bbR$ are proper, lower semi-continuous and bounded from below.
\item The linear operator $A$ is surjective.
\item For each $i = 1, 2, \cdots, n$, $H_i$ is smooth differentiable and there exists an $L > 0$ such that 
$$
\| \nabla H_i(x) - \nabla H_i(y) \| \leq L \| x - y\|, ~~~~~\forall~x,~y \in \mathbb{R}^d .
$$
\end{enumerate}
\end{assumption}

For convergence analysis of non-convex optimization problem and algorithms, the key element is finding a stability function for which decent property can be obtained. For our specific case, we find the following stability function
\begin{equation}\label{Psi}
\begin{aligned}
\Psi_{t} 
&=  L_{\beta}(x_t,z_t,u_t) + C_0 \| A^{T} (u_t - u_{t-1}) + \sigma B (x_t - x_{t-1})\|^2 + \sfrac{1}{\rho}\Pa{ \sfrac{32}{\sigma\beta\lambda_{m}} + \sfrac{\eta}{2} } \Upsilon_{t}\\
&\qquad +  \sfrac{16}{\sigma\beta\lambda_{m}}  \| \tnabla H(x_{t-1}) - \nabla H(x_{t-1}) \|^2  + C_3 \| x_{t} - x_{t-1} \|^2,
\end{aligned}
\end{equation}
where $\eta, C_2>0$ are some positive constants, and  
$$
\lambda_m = \lambda_{min}(AA^T), ~~~~~B \eqdef \tau Id - \beta A^TA , 
$$
$$
C_0 \eqdef \sfrac{4(1-\sigma)}{\sigma^2 \beta \lambda_{m}} \geq 0 , \enskip
 C_1\eqdef \sfrac{8(\sigma\tau + L)^2}{\sigma\beta\lambda_{m}} > 0 
\enskip
\textrm{and}
\enskip
C_3 = \Pa{ \sfrac{32}{\sigma\beta\lambda_{m}} + \sfrac{\eta}{2} } (V_1 + \sfrac{V_{\Upsilon}}{\rho}) + C_2 + C_1.
$$
We have the following descent property for $\Psi_{t}$ for sequences generated by Algorithm \ref{alg:SADMM}.

\begin{theorem}\label{decrease}
For problem \eqref{model} and Algorithm \ref{alg:SADMM}, suppose that $2\tau \geq \beta \|A\|^2$ and Assumption \ref{assum} holds.  Let  $\{( x_t, z_t, u_t) \}_{t \geq 0}$ be a sequence generated by Algorithm \ref{alg:SADMM} using variance-reduced gradient estimators. 
Then for any $t \geq 1$ there holds
\begin{equation}\label{dec-relation}
\mathbb{E}_t [ \Psi_{t+1} + \tilde{\eta}  \| x_{t+1} - x_t \|^2 + C_2\| x_t - x_{t-1} \|^2 + \sfrac{1}{\sigma\beta} \| u_{t+1} - u_t \|^2 ] \leq \Psi_t,
\end{equation}
where $\Psi_{t}$ is given in \eqref{Psi}, $C_2>0$ and  
$$
\tilde{\eta} =\tau - \sfrac{L + \beta \| A \|^2}{2} - \sfrac{4\sigma\tau^2}{\beta\lambda_{m}} - \sfrac{8(\sigma \tau + L)^2}{\sigma\beta\lambda_{m}} -  \sfrac{1}{2\eta} - \Pa{ \sfrac{64}{\sigma\beta\lambda_{m}} + \eta } (V_1 + \sfrac{V_{\Upsilon}}{\rho}) - C_2. 
$$
Moreover, if $\tilde{\eta}  >0$, then we have the following finite summand in expectation
\begin{equation}\nonumber
\sum_{t=0}^{\infty} \mathbb{E} [\| x_{t+1} - x_t \|^2 + \| u_{t+1} - u_t \|^2 + \| z_{t+1} - z_t \|^2] < +\infty.
\end{equation}
\end{theorem}
{{
\begin{proof}
Below we provide a sketch of the proof while details can be found in Appendix \ref{app1}.

{\bf Firstly}, according to the optimization condition of each subproblem in Algorithm \ref{alg:SADMM}, we get the following basic relation (see also \eqref{ad1}),
\begin{equation}\nonumber
\begin{aligned}
&\mathbb{E}_t \big[ \Phi_{t+1} + \sfrac{16}{\sigma\beta\lambda_{m}}  \| \tnabla  H(x_{t}) - \nabla H(x_{t}) \|^2+ ( \tau - \sfrac{L + \beta \| A \|^2}{2} - \sfrac{4\sigma\tau^2}{\beta\lambda_{m}} - \sfrac{8(\sigma \tau + L)^2}{\sigma\beta\lambda_{m}})\| x_{t+1} - x_t \|^2 + \sfrac{1}{\sigma\beta} \| u_{t+1} - u_t \|^2 \big]\\
&\leq \Phi_t +  \sfrac{16}{\sigma\beta\lambda_{m}}  \| \tnabla  H(x_{t-1}) - \nabla H(x_{t-1}) \|^2 + ( \sfrac{32}{\sigma\beta\lambda_{m}} + \sfrac{\eta}{2} ) \mathbb{E}_t \| \tnabla  H(x_t) - \nabla H(x_t) \|^2 + \sfrac{1}{2\eta} \mathbb{E}_t \| x_{t+1} - x_t \|^2.
\end{aligned}
\end{equation}
We can see that the above inequality contains $\mathbb{E}_t \| \tilde{\nabla} H(x_t) - \nabla H(x_t) \|^2$, therefore it is necessary to estimate this term.\\
{\bf Secondly}, we use the MSE bound of $\mathbb{E}_t \| \tilde{\nabla} H(x_t) - \nabla H(x_t) \|^2$ and the geometric decreasing property of $\Upsilon_t$ in Definition \ref{var-redu}, then we get the descent property for  stability function $\Psi_t$ (see also \eqref{eq31}),
\begin{equation}\nonumber
\mathbb{E}_t [ \Psi_{t+1} + \tilde{\eta}  \| x_{t+1} - x_t \|^2 + \sfrac{1}{\sigma\beta} \| u_{t+1} - u_t \|^2  + C_2 \| x_t - x_{t-1} \|^2 ] \leq \Psi_t.
\end{equation}
{\bf Finally}, applying the full expectation operator to \eqref{eq31} and combining the lower boundedness of $\{ \Psi_t \}_{t \geq 1}$ (see Lemma \ref{phi-bound}), we get the conclusion.
\end{proof}
}}

\begin{remark}[Positivity of $\tilde{\eta}$]\label{para}
Here we provide a short discussion on the positivity of $\tilde{\eta}$. Apparently, if $\tilde{\eta} > 0$, then
\begin{equation}\label{eqe13}
-2\tilde{\eta} = \sfrac{24\sigma\tau^2}{\beta\lambda_{m}} - 2\Pa{ 1 - \sfrac{16L}{\beta\lambda_{m}} } \tau + \sfrac{16L^2}{\sigma\beta\lambda_{m}} + L + \beta\|A\|^2 + \sfrac{1}{\eta} + \Pa{ \sfrac{128}{\sigma\beta\lambda_{m}} + 2\eta}(V_1 + \sfrac{V_{\Upsilon}}{\rho}) + 2C_2< 0 ,
\end{equation}
where the RHS of the equal sign is a quadratic function of $\tau$. 
The reduced discriminant of the quadratic function of $\tau$ then reads  
\begin{equation}\label{eqe14}
\Delta_{\tau}^{\prime} 
\eqdef \Pa{  1 - \sfrac{16L}{\beta \lambda_{m}} }^2 - \sfrac{24\sigma}{\beta\lambda_{m}} \bPa{ \sfrac{16L^2}{\sigma\beta\lambda_{m}} + L + \beta\|A\|^2 + \sfrac{1}{\eta} + (\sfrac{128}{\sigma\beta\lambda_{m}} + 2\eta)(V_1 + \sfrac{V_{\Upsilon}}{\rho}) + 2C_2 } > 0. 
\end{equation}
which, after expansion, is another quadratic function of $\beta$. Since $\beta > 0$, we then get from above
\begin{equation}\label{eqe15}
(1 -24\sigma \kappa(AA^T))\beta^2 - 2\bPa{ 4 + 3\sigma + \sfrac{3\sigma}{\eta L} + \sfrac{6\sigma \eta}{L} (V_1 + \sfrac{V_{\Upsilon}}{\rho}) + \sfrac{6\sigma C_2}{L} } \nu \beta - 8 \nu^2 - \sfrac{192\nu^2}{L^2}(V_1 +  \sfrac{V_{\Upsilon}}{\rho}) > 0.
\end{equation}
Note that the reduced discriminant of the quadratic function in $\beta$ in the above relation  is always greater than $0$ provided that  $0 < \sigma < \sfrac{1}{24 \kappa(AA^T)}$.  Hence, if $\beta > \beta_+$ with $\beta_+$ being a larger root of the quadratic function in $\beta$,  then \eqref{eqe15}  holds and hence \eqref{eqe14} is  also true.  On the other hand,  it is not difficult  to compute that  the larger root $\tau_+$ of the quadratic equation on  $\tau$ in \eqref{eqe13} is equal to $\sfrac{\beta \lambda_{m}}{24\sigma} (1 - \sfrac{4\nu}{\beta} + \sqrt{\Delta_{\tau}^{\prime}})$ where $ \nu \eqdef \sfrac{4L}{\lambda_m} > 0$.  Next we show  $\tau_+ > \sfrac{\beta\|A\|^2}{2} >0$.  This is easy to see,  according to \eqref{eqe14},  we get 
\begin{equation}\label{eta2} 
\sfrac{\beta \|A\|^2}{2} 
< \sfrac{\beta \lambda_m}{24\sigma} \Pa{ 1 - \sfrac{4\nu}{\beta} } 
\Leftrightarrow 1 - \sfrac{4\nu}{\beta} - 12\sigma \kappa(AA^T) >0,
\end{equation}
which follows from $\beta > \beta_+$ immediately. 
\end{remark}

For convenience, let's define:
\begin{equation}\label{not1}
\begin{aligned}
&\Phi(x, z, u, x^{\prime}, u^{\prime}) \eqdef L_{\beta} (x, z, u) + \sfrac{4(1- \sigma)}{\sigma^2\beta\lambda_{m}} \| A^T ( u- u^{\prime} ) + \sigma B(x - x^{\prime})  \|^2 + \sfrac{8 (\sigma \tau + L)^2}{\sigma\beta\lambda_{m}} \| x - x^{\prime} \|^2\\
&\Phi_t \eqdef \Phi(x_t, z_t, u_t, x_{t-1}, u_{t-1}), ~~~~X_t = (x_t, z_t, u_t, x_{t-1}, u_{t-1}), ~~~~X^* = (x^*,z^*, u^*, x^*, u^*),
\end{aligned}
\end{equation}
where $\{ x_t, z_t, u_t\}$ is the sequence generated by Algorithm \ref{alg:SADMM}.

Next, under the condition that the objective function satisfies the KL property (see \cite{BST, ABRS}), we make full use of the decreasing property of stability functional  $\Psi_t$  and the boundedness of  $\dist(0, \partial \Phi(X_t))$ (see Lemma \ref{subgrad-bound}) to prove the following convergence results. See Appendix \ref{app1} for details of the proof.  

\begin{theorem}\label{convergence}
For problem \eqref{model} and Algorithm \ref{alg:SADMM}, suppose  $2\tau \geq \beta \|A\|^2$ and that Assumption \ref{assum} holds, and $\Phi$ is a semialgebraic function with $KL$ exponent $\theta \in [0, 1). $ Let $\{ X_t \}_{t=0}^{\infty}$ be a bounded sequence of iterates of Algorithm \ref{alg:SADMM} using a variance-reduced gradient estimator and $\tilde{\eta} > 0$. Then either $X_t$ is a critical point after a finite number of iterations, or $\{ X_t \}_{t=0}^{\infty}$ almost surely satisfies the finite length property in expectation:
\begin{equation}\label{eq59}
\sum_{t=0}^{\infty} \mathbb{E} \| x_{t+1} - x_t \| < \infty ,\quad
\sum_{t=0}^{\infty} \mathbb{E} \| u_{t+1} - u_t \| < \infty
\quad\textrm{and}\quad
\sum_{t=0}^{\infty} \mathbb{E} \| z_{t+1} - z_t \| < \infty.
\end{equation}
Moreover, there exists an iteration $m$ such that for all $i > m$,
\begin{equation}\nonumber
\begin{aligned}
&\sum_{t = m}^{i} \mathbb{E} \| x_{t+1} - x_t \| + \mathbb{E} \| x_{t} - x_{t-1} \|  + \mathbb{E} \| u_{t+1} - u_t \| \\
&\leq \sqrt{\mathbb{E} \| x_m - x_{m-1} \|^2} + \sqrt{\mathbb{E} \| x_{m-1} - x_{m-2} \|^2}+ \sqrt{\mathbb{E} \| u_m - u_{m-1} \|^2} + \sfrac{(2\sqrt{2} + 1)\sqrt{s}}{3K_1 \rho} \sqrt{\mathbb{E} \Upsilon_{m-1}} + K_3 \Delta_{m, i+1},
\end{aligned}
\end{equation} 
where 
$$
K_1 = p + \sfrac{2\sqrt{s} \sqrt{V_r}}{\rho},~~~~ K=\min\{\tilde{\eta}, \sfrac{1}{\sigma\beta}, C_2\},~~~ K_3 = \sfrac{(4\sqrt{2} + 2)K_1}{3K},
$$
and $\Delta_{m,n} \eqdef \phi (\mathbb{E}[\Psi_{m} - \Phi_{m}^*]) - \phi(\mathbb{E} [\Psi_n - \Phi_{n}^*])$. Here $p$ is a constant in Lemma \ref{subgrad-bound}.
\end{theorem}
\begin{proof}
{Below we also only provide a sketch and refer to Appendix \ref{app1} for details. 
First of all, we show an upper estimate of subgradients of the function $\Phi$ at $(x_t, z_t, u_t, x_{t-1}, u_{t-1})$ for every $t \geq 1$ in Lemma \ref{subgrad-bound}. Next, we split the proof into three steps.\\
{\bf Step 1.} We show that $\Phi$ also satisfies the KL property with exponent $\theta \in (0,1)$ in expectation (see Theorem \ref{KL}), and we prove that the stability function $\Psi$ satisfies the following inequality:
\begin{equation}\nonumber
\phi^{\prime}(\mathbb{E}[ \Psi_t - \Phi_t^*] ) {\bf C_t} \geq 1, ~~~\forall t >m.
\end{equation}
{\bf Step 2.} By using the concavity of $\phi$ and the decreasing property of $\Psi_t$, we have (see also \eqref{eq78})
\begin{equation}\nonumber
\Delta_{t,t+1} {\bf C_t} \geq K \mathbb{E} \| x_{t+1} - x_t \|^2 + K \mathbb{E} \| x_{t} - x_{t-1} \|^2+ K \mathbb{E} \| u_{t+1} - u_t \|^2.
\end{equation}
{\bf Step 3.} Combining the above inequality, Young’s inequality and Jensen’s inequality, we have $\{ \| x_{t+1} - x_t \|, \| u_{t+1} - u_t \|, \| z_{t+1} - z_t \| \}$ is summable and end the proof.}
\end{proof}



\begin{remark}
Follow the result of \cite{DTLDS}, convergence rates can also be obtained under KL exponent. However, we refer to \cite{DTLDS} for discussions since the obtained rates are no different from the standard deterministic scenario, that is for the function $\Phi$, finite termination can be reached if the exponent $\theta = 0$, linear convergence for $\theta \in ]0, 1/2]$ and sub-linear rate for $\theta \in ]1/2, 1[$. 

\end{remark}

%
%
%
%
%

\section{ Numerical experiments}
\label{sec:examples}

In this section, we provide numerical experiments to verify the performance of Algorithm \ref{alg:SADMM}. Three problems are considered:graph-guided Fussed Lasso, wavelet frame based 2D CT reconstruction and  TV-$L_0$ for 3D CT reconstruction. All experiments are run in MATLAB R2019a on a desktop equipped with a 4.0GHz 8-core AMD processor and 16GB
memory.
\subsection{Graph-guided Fused Lasso}
\label{4.1}
We first consider a binary classification problem which combines  correlations between features. Given a set of training samples $\{ (a_i, b_i) \}_{i=1}^{n}$, where $a_i \in \mathbb{R}^m$ and $b_i \in \{ -1, +1 \}$ for $i \in \{ 1, 2, \cdots, n\}.$ We solve this problem by using the following model, called the graph-guided fused lasso \cite{KSX}:
\begin{equation}\label{Flasso}
\min_{x} \sfrac{1}{n} \sum_{i=1}^{n} f_i(x) + \lambda_1\|Ax\|_1,
\end{equation}
where $f_i(x) = \sfrac{1}{1 + \mbox{exp}(b_i a_i^{T} x)}$ is the sigmoid loss function which is nonconvex and smooth, and $\lambda_1$ is the regularization parameter. The matrix $A = [G; I]$ and $G$ is obtained by sparse inverse covariance matrix estimation \cite{FHT,HSDR}.

In this experiment, we set $H(x) = \sfrac{1}{n} \sum_{i=1}^{n} f_i(x)$ and $F(Ax) = \lambda_1\|Ax\|_1$.  We consider the four publicly available datasets \cite{data} shown in Table \ref{GLS1}.

\begin{table}[htbp]\footnotesize
\centering
\begin{tabular}{|c|c|c|c|c|c|}
\hline
datasets &  $\sharp$training & $\sharp$test &  $\sharp$features & $\sharp$classes \\ \cline{1-5}
a8a  &11348 & 11348 & 123 & 2 \\ \cline{1-5}
ijcnn & 17500  & 17500  & 22 & 2  \\\cline{1-5}
mushrooms & 4062 & 4062 & 112 & 2  \\\cline{1-5}
w3a & 22418 & 22419 & 300 & 2  \\\cline{1-5}
\end{tabular}
\caption{Real datasets for graph-guided fused lasso.}\label{GLS1}
\end{table}
We evaluate the following several algorithms when applied to the model \eqref{Flasso}: SADMM with SARAH estimator, SADMM with SAGA estimator, ADMM with SVRG estimator \cite{HC}, ADMM with SGD estimator \cite{HC} and the deterministic ADMM \cite{LPADMM}. We fix the parameter $\lambda_1 = 1e-5, \sigma = 0.95$ and $\beta = 1$. We point out that the selection of parameter $\eta$ in \cite{HC} and the selection of $\tau$ in Algorithm \ref{alg:SADMM} are conservative.  We choose the parameters $\eta$ and $\tau$ to guarantee the convergence, and  achieve the optimal result at the same time. Since we have no requirement for the batch size $b$, we choose the batch size in each algorithm that gives the optimal experimental result.  We  take  $x^0 = 0$ for the initialization.

\begin{figure}[!ht]
	\centering
	\subfloat[{\tt a8a}]{ \includegraphics[width=0.4\linewidth]{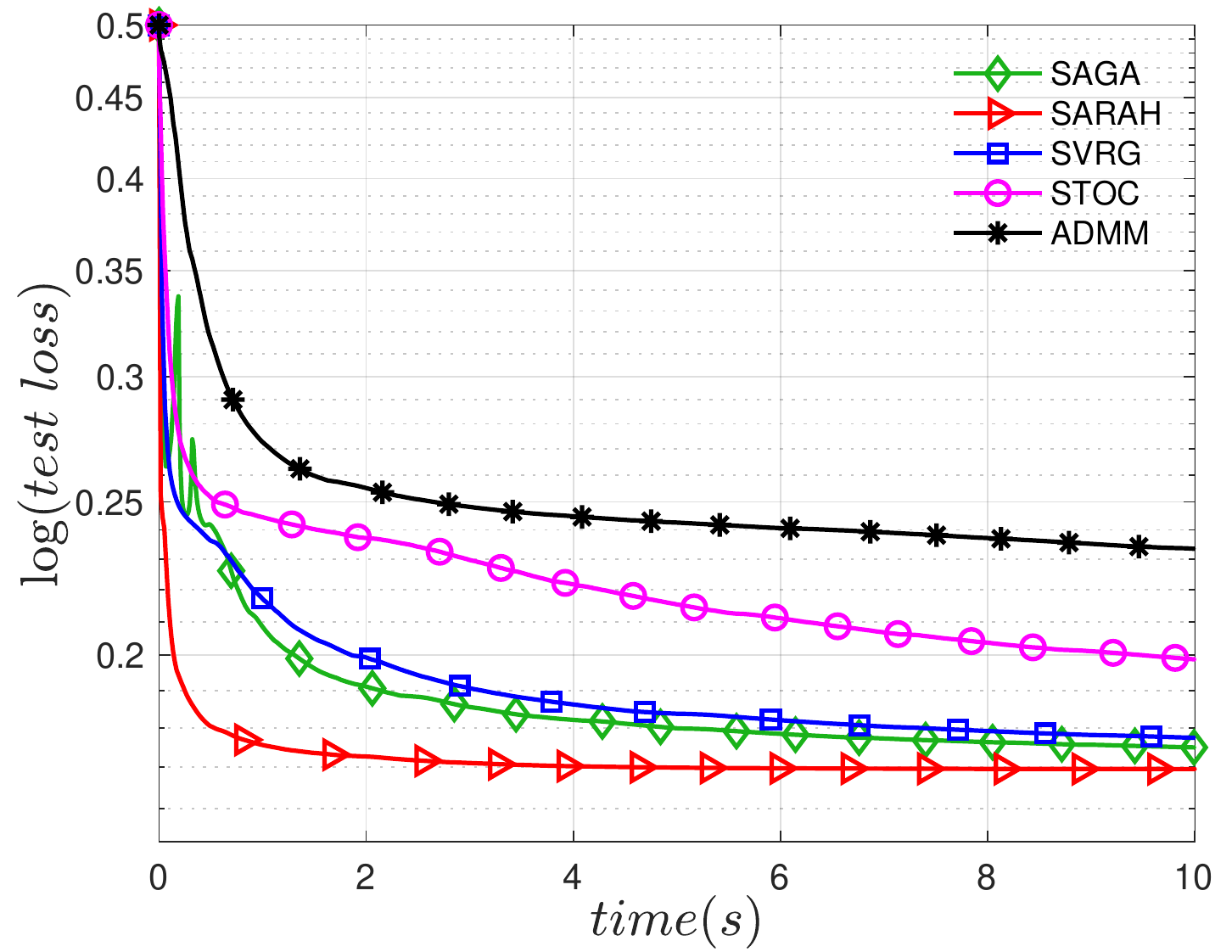} } 		
	\subfloat[{\tt ijcnn}]{ \includegraphics[width=0.4\linewidth]{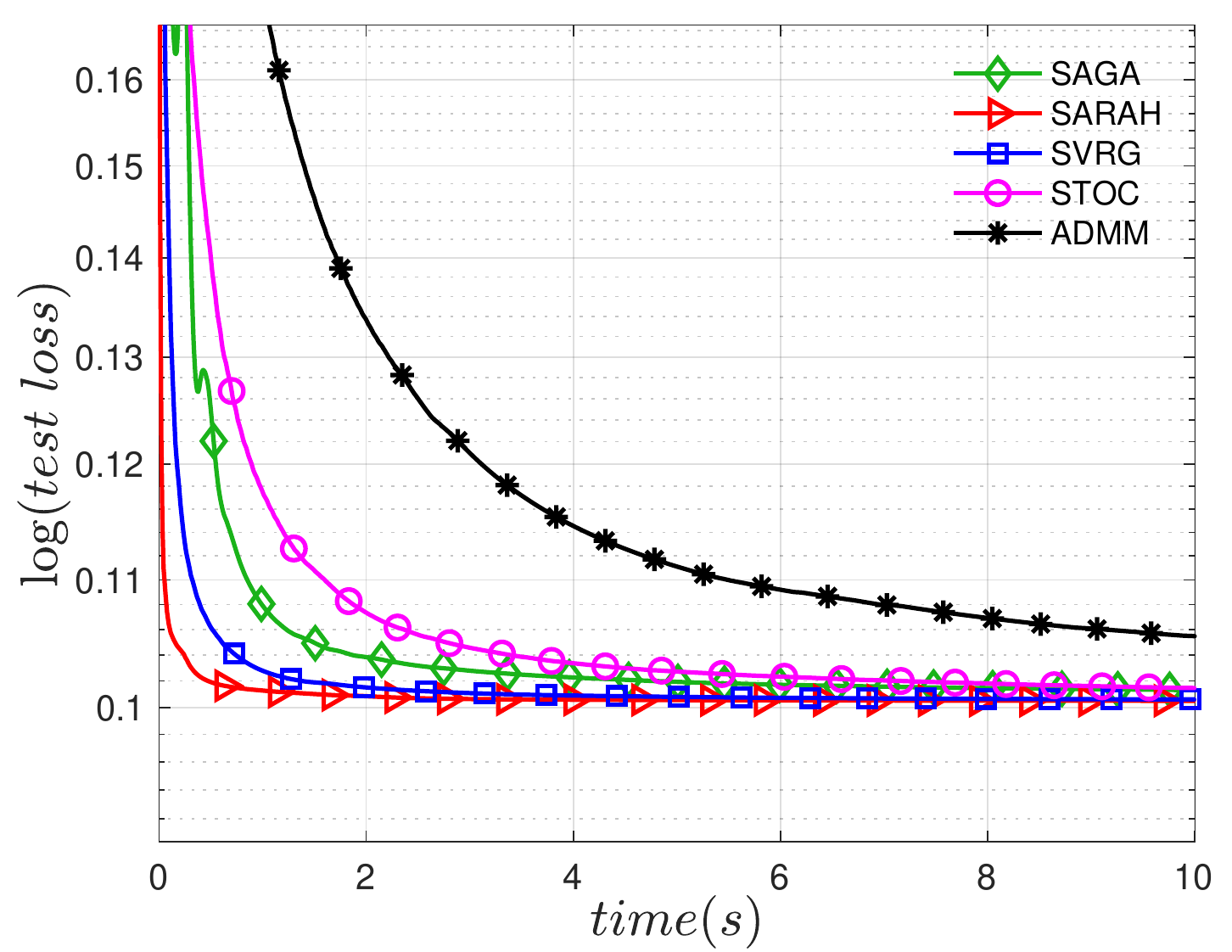} } 	\\[-2mm]	
	\subfloat[{\tt mushrooms}]{ \includegraphics[width=0.4\linewidth]{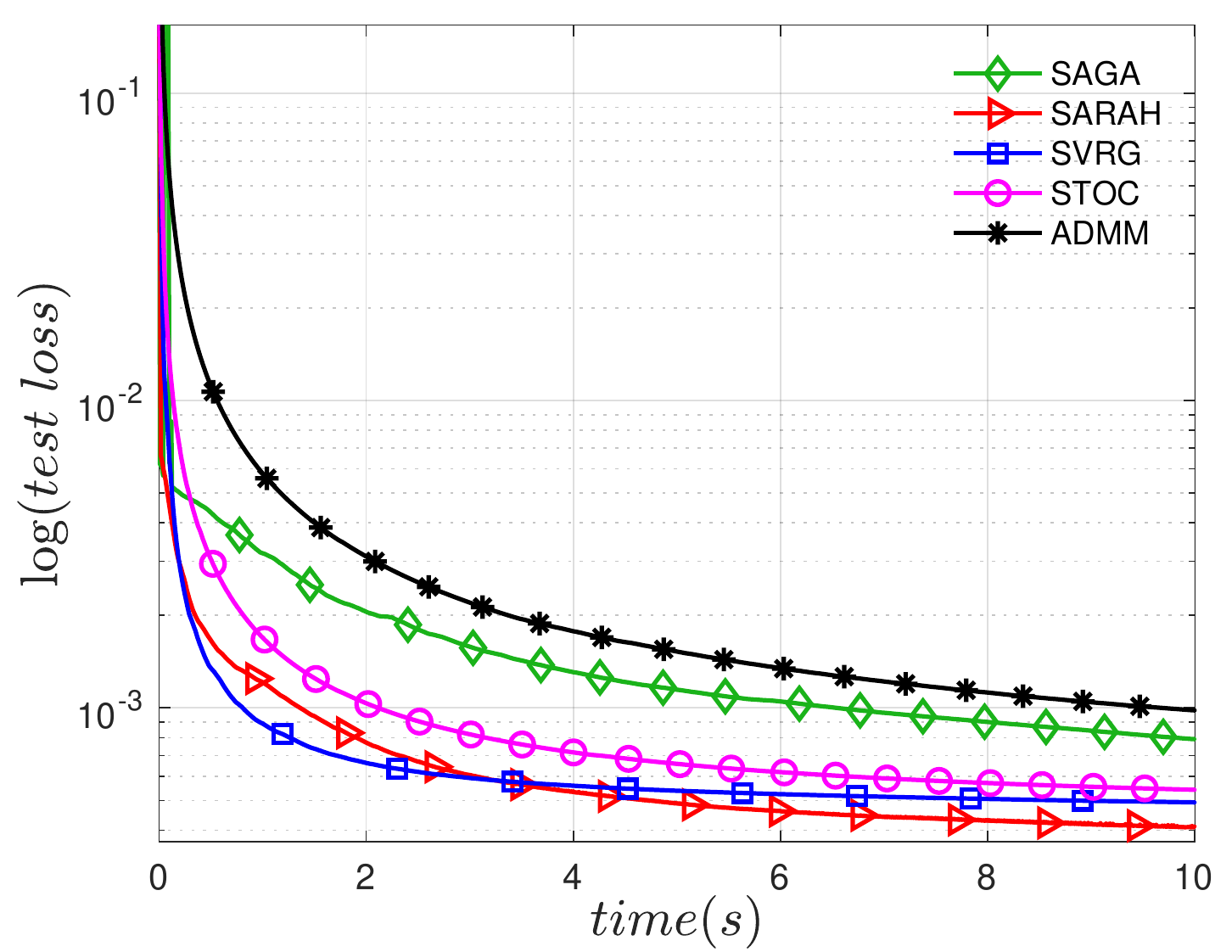} } 
	\subfloat[{\tt w3a}]{ \includegraphics[width=0.4\linewidth]{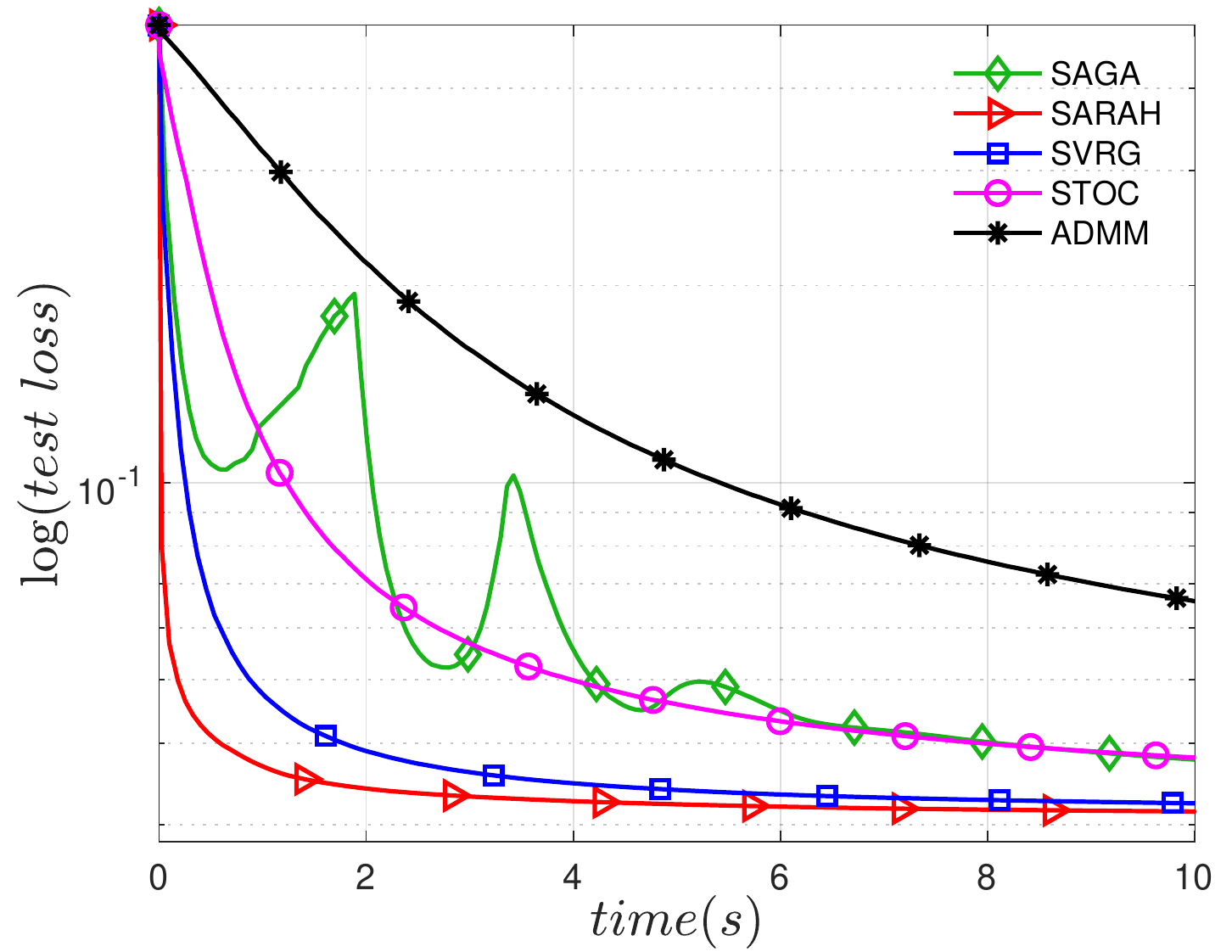} } 	\\[-2mm]	
	\caption{Test loss for different methods with the same CPU-time.} 
	\label{GLS2}
\end{figure}

%

In Figure \ref{GLS2}, we give the results of different methods for test loss with the same running time. Figure \ref{GLS3} shows that the results of different algorithms for solving the problem \eqref{Flasso} by using the same number of epoch, where each epoch estimates $n$ component gradients. We can see that the SARAH-ADMM always has the lowest test-loss and it also faster than SAGA-ADMM, SVRG-ADMM, STOC-ADMM and ADMM. We also note that SVRG-ADMM and SAGA-ADMM have the similar behaviors and both algorithms  are faster than the deterministic ADMM.

\begin{figure}[!ht]
	\centering
	\subfloat[{\tt a8a}]{ \includegraphics[width=0.4\linewidth]{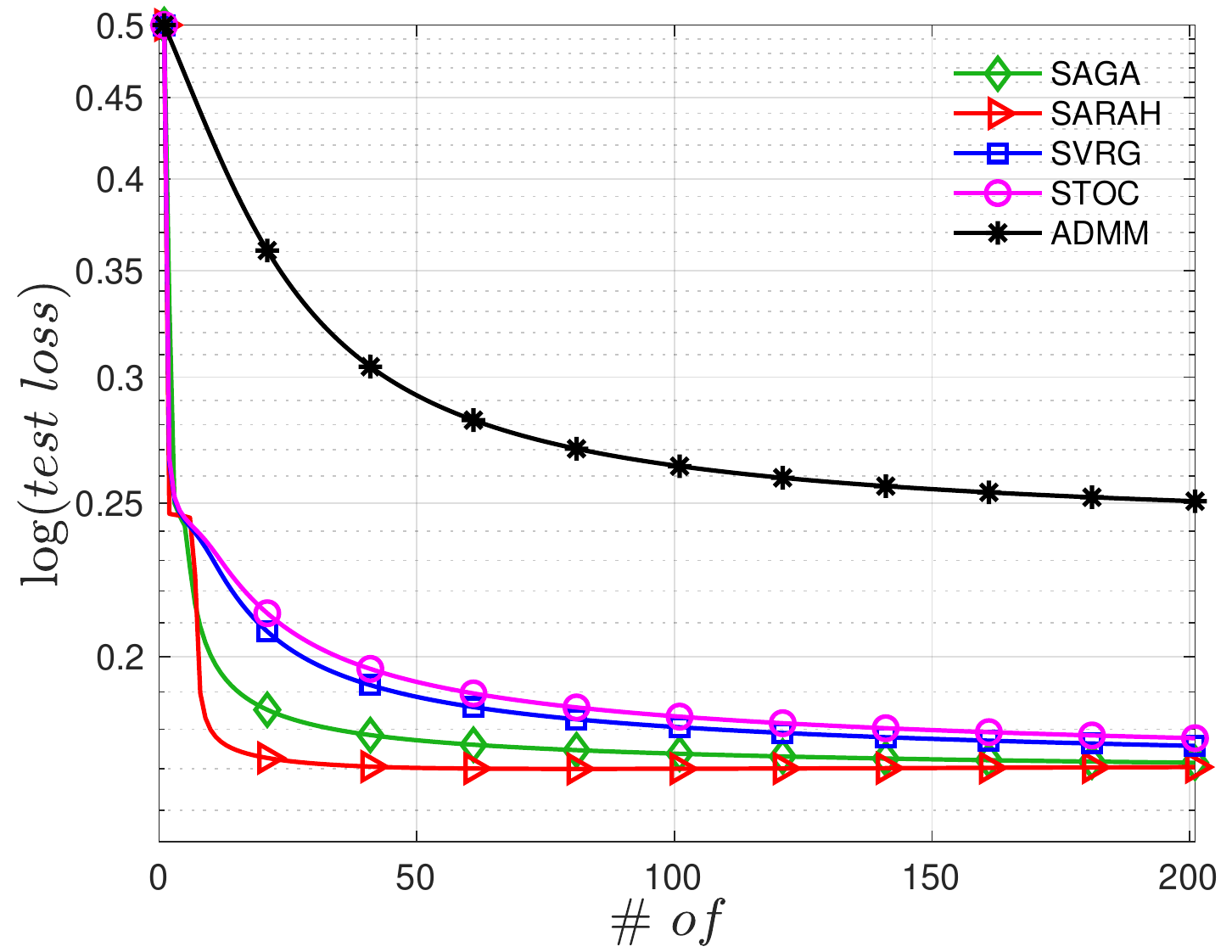} } 		
	\subfloat[{\tt ijcnn}]{ \includegraphics[width=0.4\linewidth]{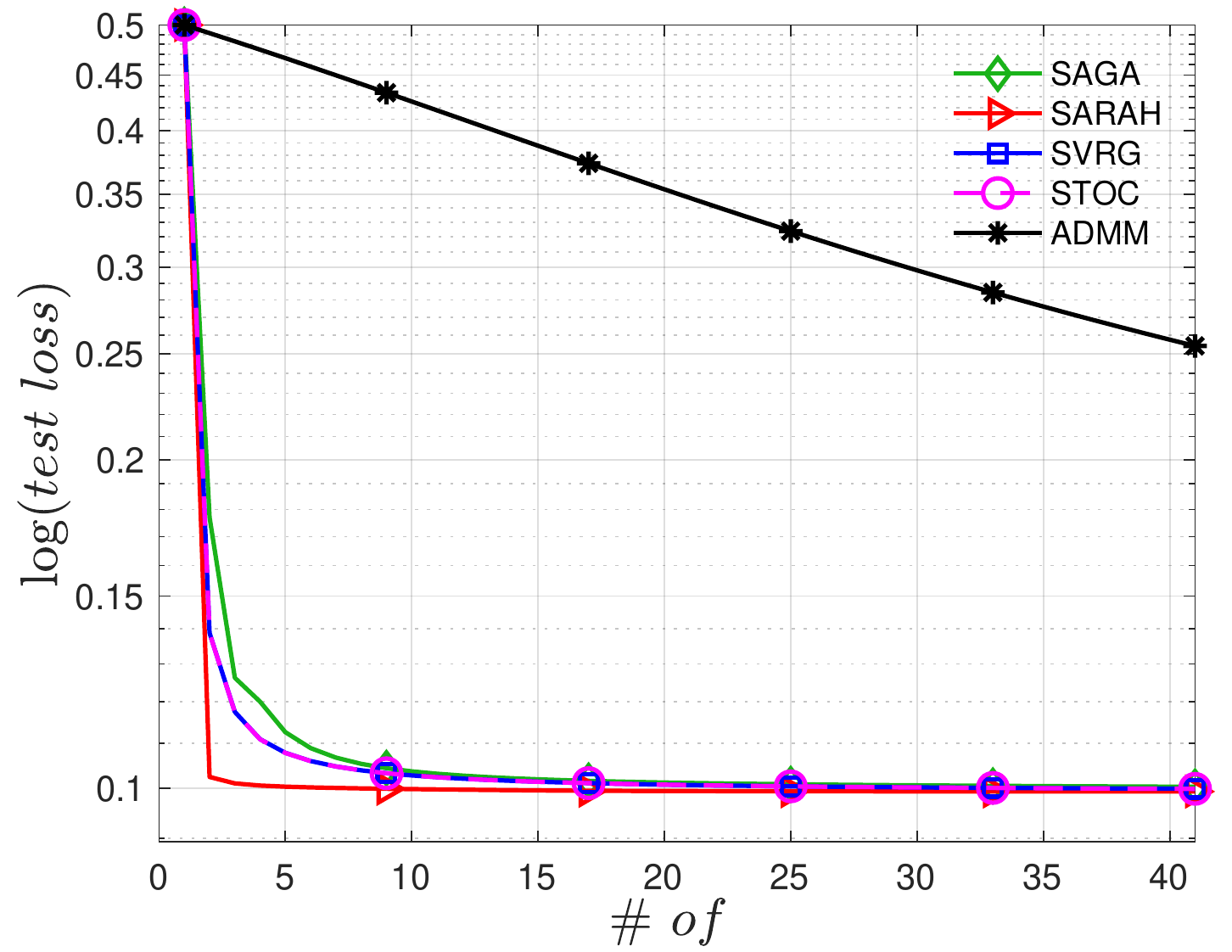} } 	\\[-2mm]	
	\subfloat[{\tt mushrooms}]{ \includegraphics[width=0.4\linewidth]{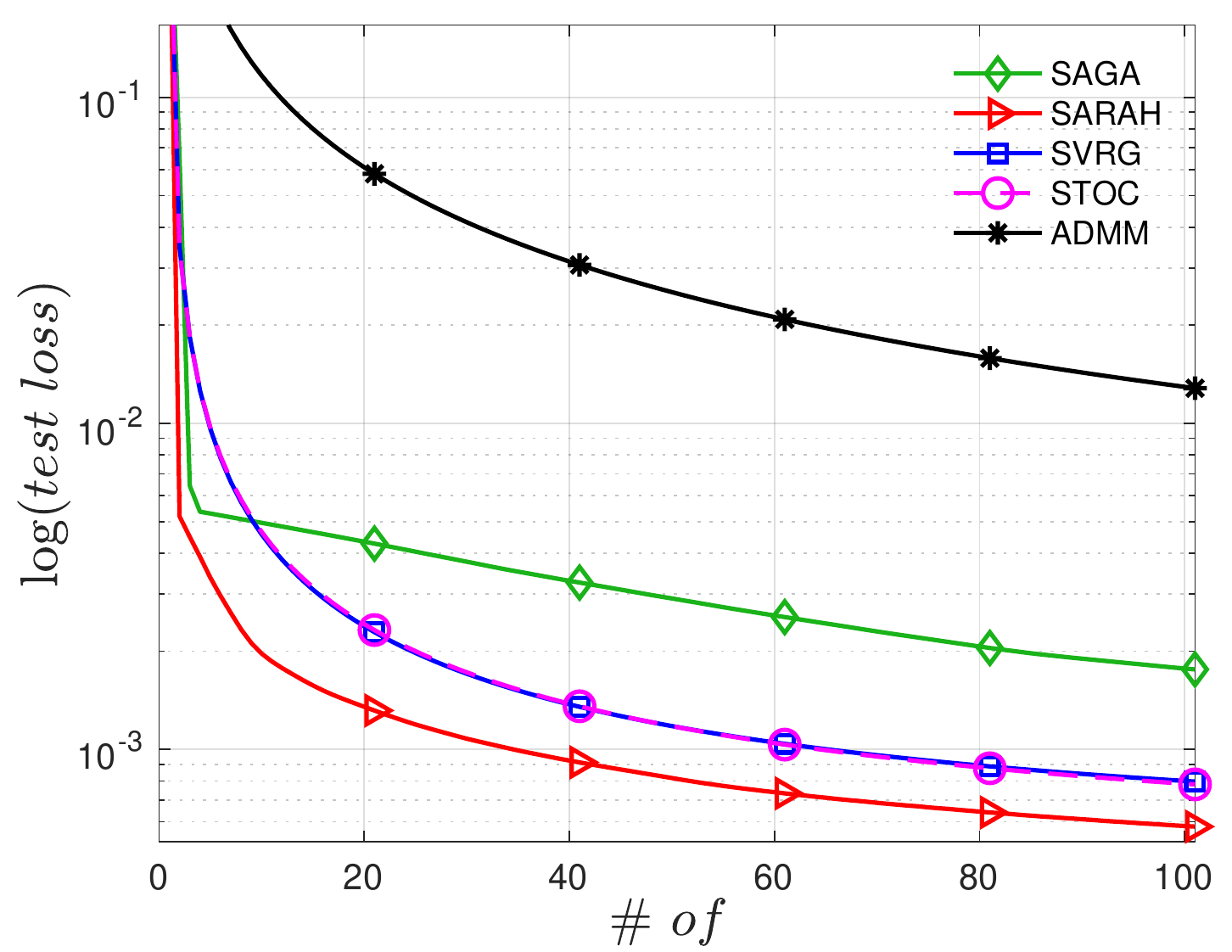} } 
	\subfloat[{\tt w3a}]{ \includegraphics[width=0.4\linewidth]{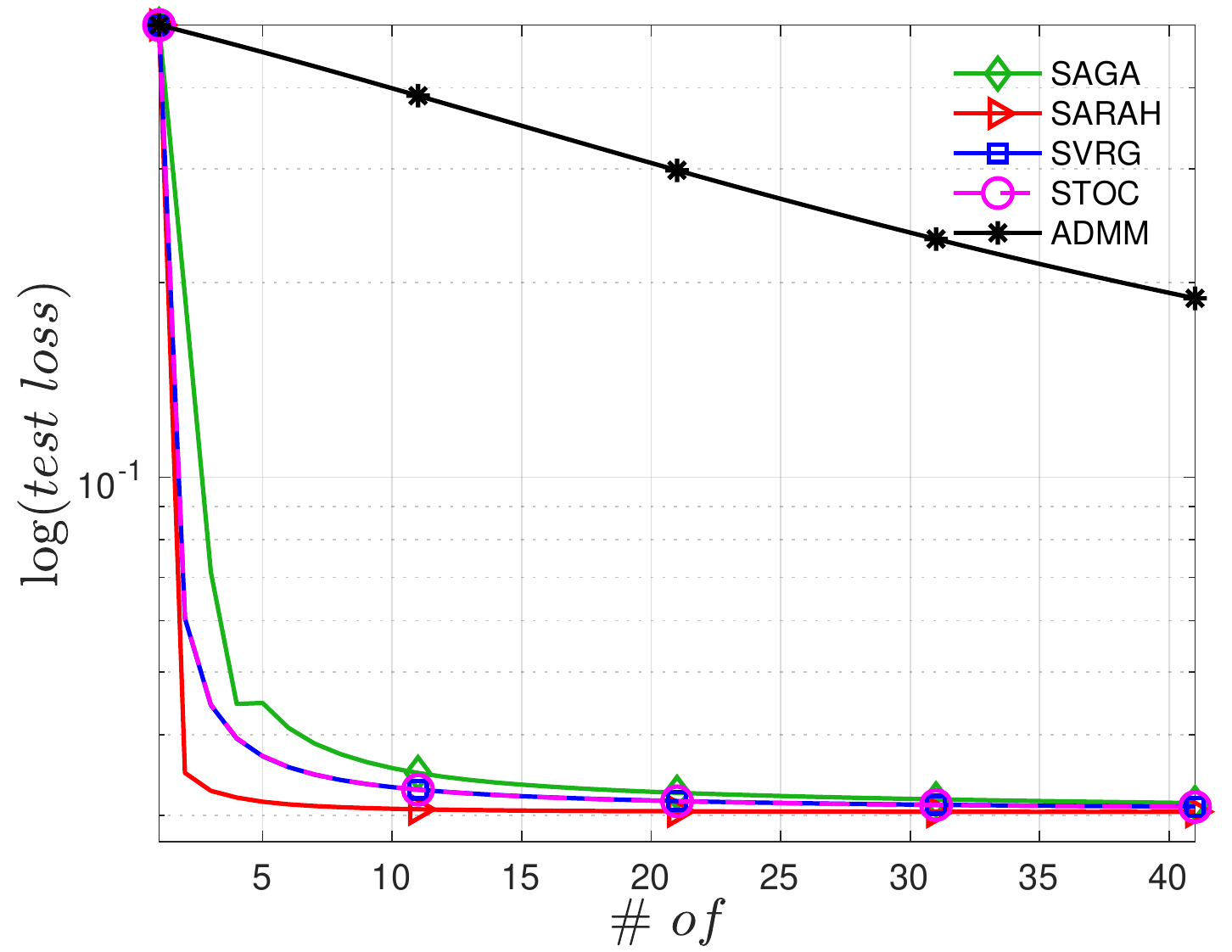} } 	\\[-2mm]	
	\caption{Test loss for different methods with the same number of epoch.} 
	\label{GLS3}
\end{figure}

\subsection{ Wavelet Frame Based 2D CT Reconstruction}
\label{4.2}
In the past decade, wavelet frame based models have been widely used for image denoising, reconstruction and restoration \cite{ZDL}. Wavelet frame based approaches utilized the observation that images can often be sparsely approximated by properly designed wavelet frames. At the same time, a wavelet frame based $L_0$ minimization model has presented its advantages \cite{DZ}. In this subsection, we use the wavelet frame based $L_0$ norm as a regularization term for 2D CT reconstruction problem:
\begin{equation}\nonumber
\min_{x} \sfrac{1}{n} \sum_{i=1}^n ( \mathcal{R}_i x - b_i )^2 + \sum_{i} \lambda_i \| (\mathcal{W}x)_i \|_0,
\end{equation}
where $\mathcal{R}$  is the Radon operator generated by using fan beam scanning geometry \cite{G} and the number of detectors is 512. We generate the observed data in the following two cases:\\

{\bf Case 1.}  Let the number of viewers be 360, then the dimension of observed  data $b$ is 184320. We add the white noise with mean 0 variance 0.5 on the measured data $b$.\\

{\bf Case 2.} Let the number of viewers be 60, then the dimension of observed  data $b$ is 30720. We  add the white noise with mean 0 variance 0.15 on the measured data $b$.\\

We solve the above two cases by using SARAH-ADMM, SAGA-ADMM, SVRG-ADMM, STOC-ADMM and deterministic ADMM. In two cases, we set $\lambda = 1e-5$ and  $\beta = 1e-4$. In each algorithm, $\tau$ and batch size are taken to be the value that makes the experiment result best. Figure \ref{wav1} shows how the PSNR  of images changes as the running time increases.  As we can see in the magnified  image on the left, ADMM and SARAH-ADMM can achieve relatively high PSNR in the shortest time. From the magnified  image on the right and  Figure \ref{wav2}, we see that the final PSNR of SARAH-ADMM is relatively high. The SAGA-ADMM can also reach a relatively high PSNR, but its speed is relatively slow. In Figure \ref{wav2} and Figure \ref{wav3}, we give the final CT image reconstruction results for {\bf Case 1} and {\bf Case 2} respectively, from which we see that SARAH-ADMM gets better results than  other methods. \\

\begin{figure}[htbp]
\centering
\includegraphics[width=0.8\textwidth] {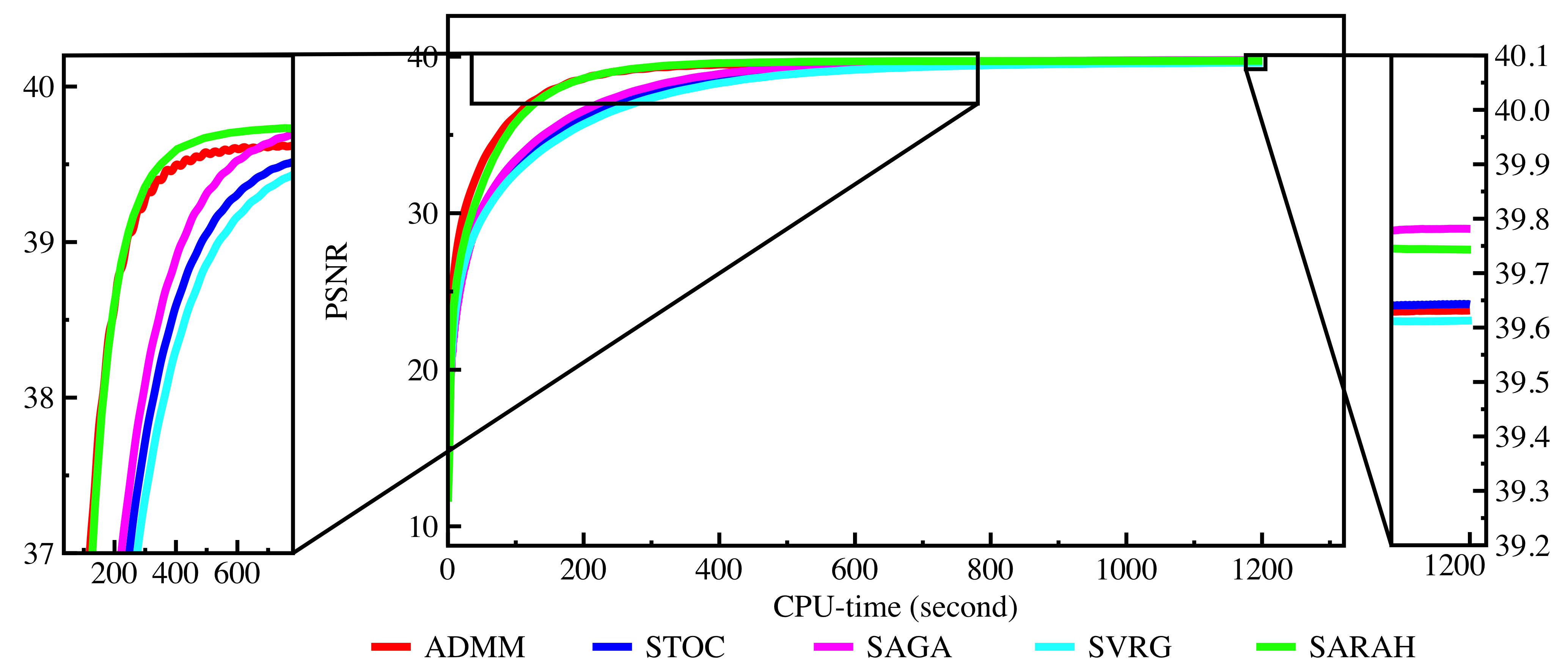}
\caption{PSNR results of different methods for {\bf Case 1}}\label{wav1}
\end{figure}

\begin{figure}[!ht]
	\centering
	\subfloat[{\tt Ground truth}]{ \includegraphics[width=0.3\linewidth]{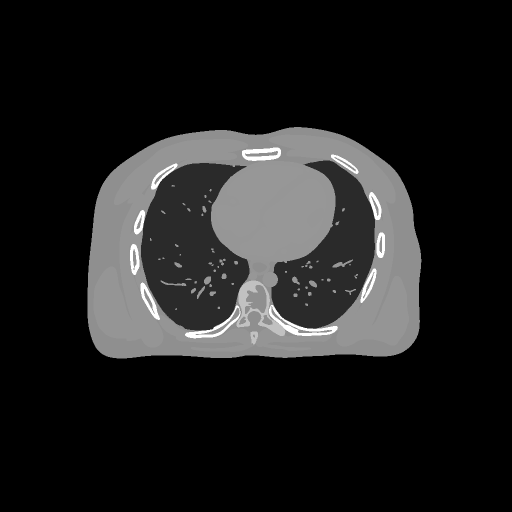} } 		
	\subfloat[{\tt STOC= 39.64}]{ \includegraphics[width=0.3\linewidth]{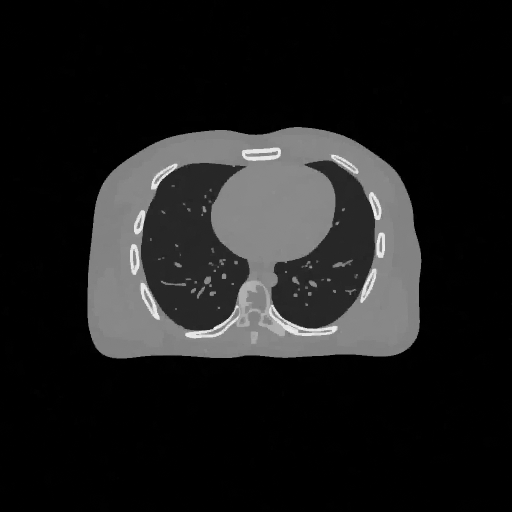} } 		
	\subfloat[{\tt SVRG= 39.61}]{ \includegraphics[width=0.3\linewidth]{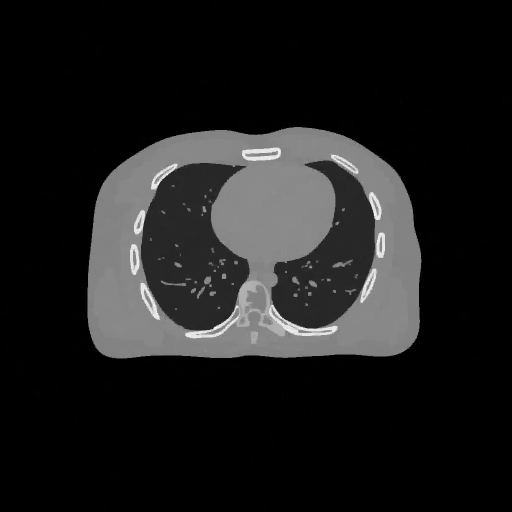} } 	\\[-2mm]	
	\subfloat[{\tt ADMM = 39.63}]{ \includegraphics[width=0.3\linewidth]{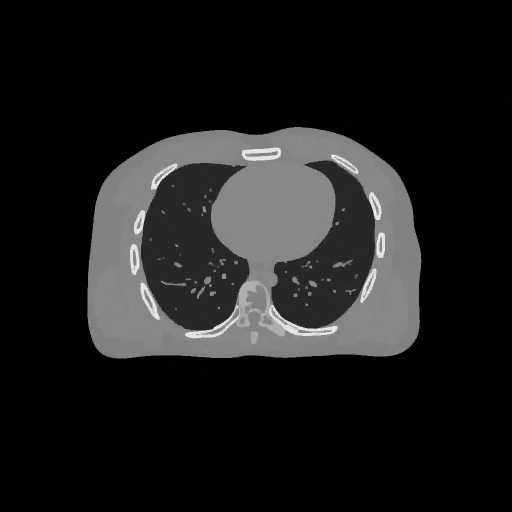} } 
	\subfloat[{\tt SAGA= 39.78}]{ \includegraphics[width=0.3\linewidth]{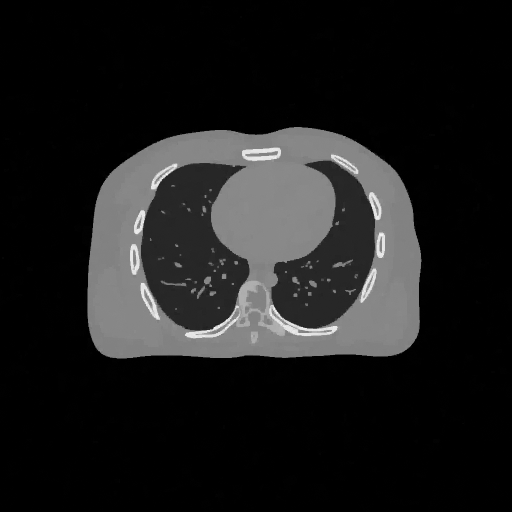} } 
	\subfloat[{\tt SARAH= 39.74}]{ \includegraphics[width=0.3\linewidth]{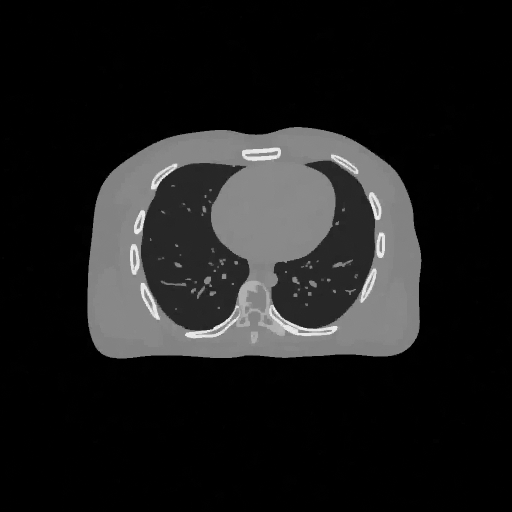} } 	\\[-2mm]	
	\caption{ Final reconstruction images of different methods for {\bf Case 1}.}
    \label{wav2}
\end{figure}


\begin{figure}[!ht]
	\centering
	\subfloat[{\tt Ground truth}]{ \includegraphics[width=0.3\linewidth]{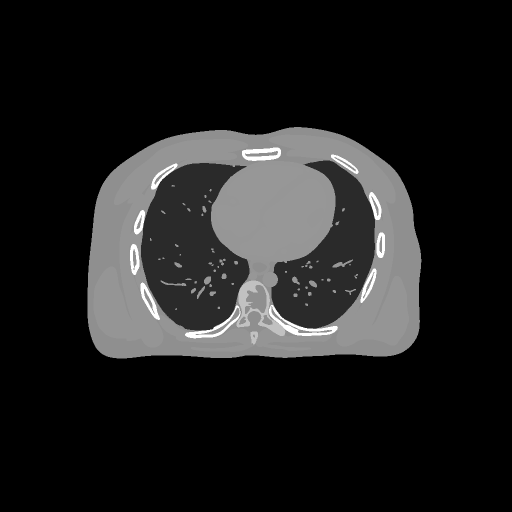} } 		
	\subfloat[{\tt STOC= 38.27}]{ \includegraphics[width=0.3\linewidth]{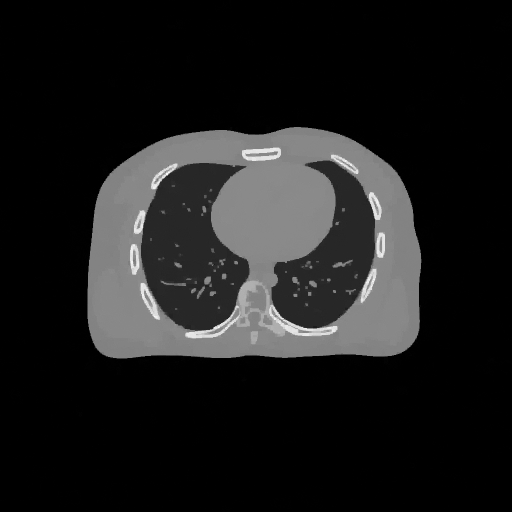} } 		
	\subfloat[{\tt SVRG= 37.99}]{ \includegraphics[width=0.3\linewidth]{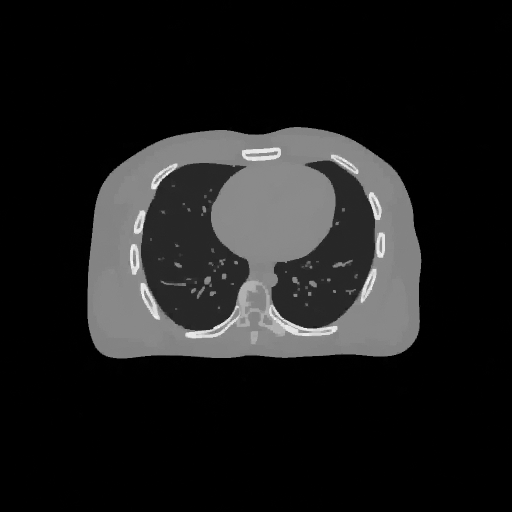} } 	\\[-2mm]	
	\subfloat[{\tt ADMM = 37.99}]{ \includegraphics[width=0.3\linewidth]{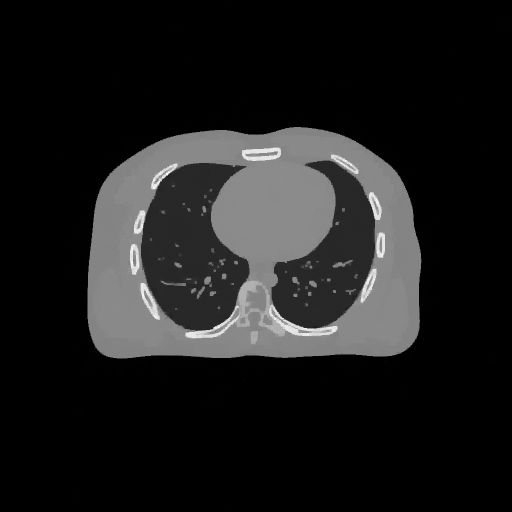} } 
	\subfloat[{\tt SAGA= 37.54}]{ \includegraphics[width=0.3\linewidth]{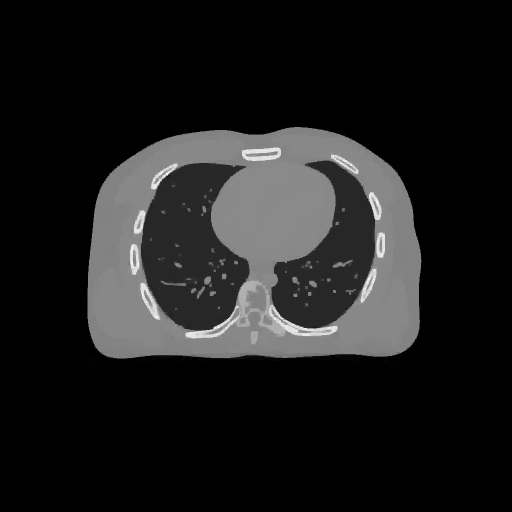} } 
	\subfloat[{\tt SARAH= 38.71}]{ \includegraphics[width=0.3\linewidth]{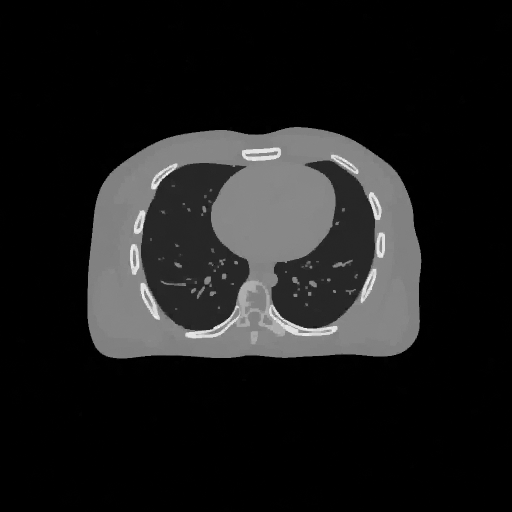} } 	\\[-2mm]	
	\caption{Final reconstruction images of different methods for {\bf Case 2}.}
    \label{wav3}
\end{figure}


\subsection{ TV-$L_0$ for 3D CT Reconstruction}
\label{4.3}
In this subsection, we consider the 3D CT reconstruction problem by using TV-$L_0$ norm as the regularization :
\begin{equation}\nonumber
\min_{x} \sfrac{1}{n} \sum_{i=1}^n ( \mathcal{R}_i x - b_i )^2 +  \lambda \| \nabla x \|_0,
\end{equation}
where $\mathcal{R}$ is the Radon transform generated by using the cone beam scanning geometry \cite{G}, the detectors plane is $512 \times 384$ and the number of viewers is 668. The size of 3D image which will be  constructed is $256 \times 256 \times 64$. Then the size of observed data $b$ is 131334144, that is to say, $n$ is considerably large. Here, we compare SARAH-ADMM with SVRG-ADMM, STOC-ADMM and ADMM. We let $\lambda = \mathcal{O}(1e-2)$, $\beta  = \mathcal{O}(1e-11)$ and $\sigma = 0.95$. The parameters $\tau$ and batch size $b$ are selected so that each algorithm can achieve the best result. Figure \ref{3d1} shows how the PSNR of 3D images changes as the running time increases. Figure \ref{3d2}, \ref{3d3} present the images and PSNR of the 15th, 55th slices of the 3D image. We can see that SARAH-ADMM and SVRG-ADMM can achieve a better PSNR in a relatively short time. However, STOC-ADMM and ADMM are relatively slow. From Figure \ref{3d2} and Figure \ref{3d3}, we also can see that the final PSNR of SARAH-ADMM and SVRG-ADMM is better. 

\begin{figure}[htbp]
\centering
\includegraphics[width=0.8\textwidth] {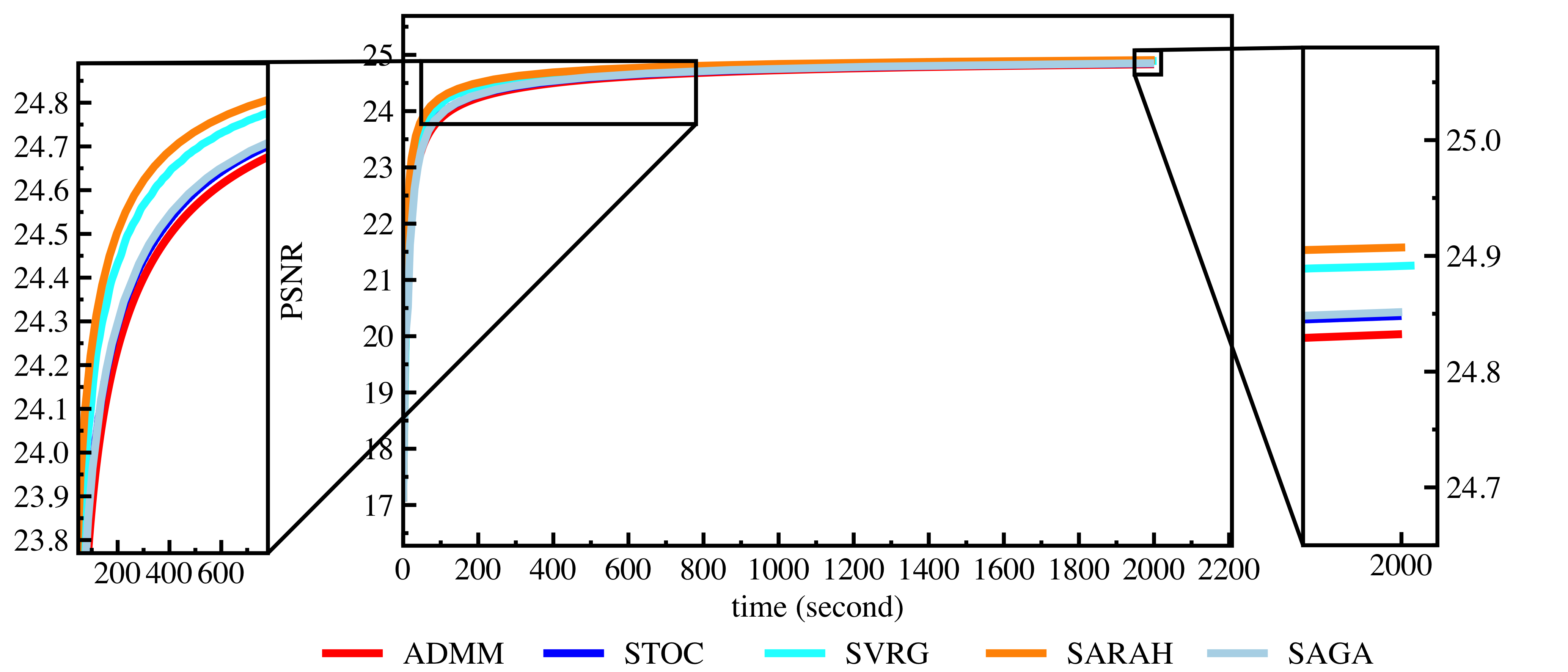}
\caption{PSNR for different methods}\label{3d1}
\end{figure}

\begin{figure}[!ht]
	\centering
	\subfloat[{\tt Ground truth}]{ \includegraphics[width=0.3\linewidth]{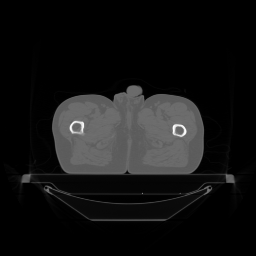} } 		
	\subfloat[{\tt STOC= 49.05}]{ \includegraphics[width=0.3\linewidth]{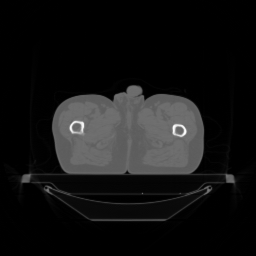} } 		
	\subfloat[{\tt SVRG= 52.37}]{ \includegraphics[width=0.3\linewidth]{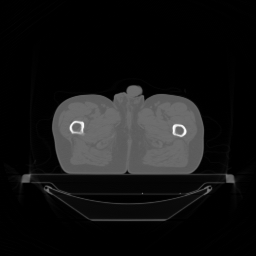} } 	\\[-2mm]	
	\subfloat[{\tt ADMM = 47.82}]{ \includegraphics[width=0.3\linewidth]{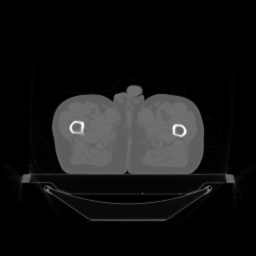} } 
	\subfloat[{\tt SAGA= 49.39}]{ \includegraphics[width=0.3\linewidth]{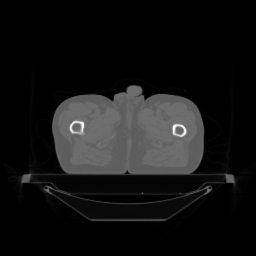} } 
	\subfloat[{\tt SARAH= 53.58}]{ \includegraphics[width=0.3\linewidth]{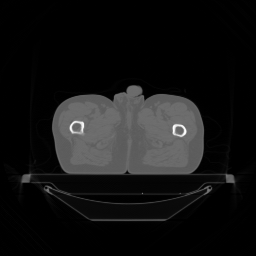} } 	\\[-2mm]	
    \caption{ Final reconstruction image of different methods for the 15th slice.}
    \label{3d2}
\end{figure}

\begin{figure}[!ht]
	\centering
	\subfloat[{\tt Ground truth}]{ \includegraphics[width=0.3\linewidth]{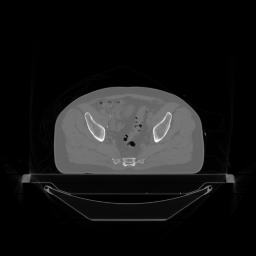} } 		
	\subfloat[{\tt STOC= 46.86}]{ \includegraphics[width=0.3\linewidth]{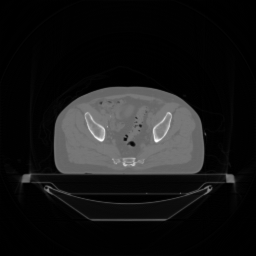} } 		
	\subfloat[{\tt SVRG= 49.99}]{ \includegraphics[width=0.3\linewidth]{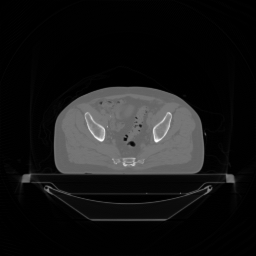} } 	\\[-2mm]	
	\subfloat[{\tt ADMM = 45.78}]{ \includegraphics[width=0.3\linewidth]{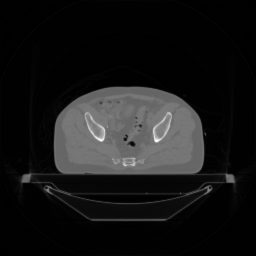} } 
	\subfloat[{\tt SAGA= 47.12}]{ \includegraphics[width=0.3\linewidth]{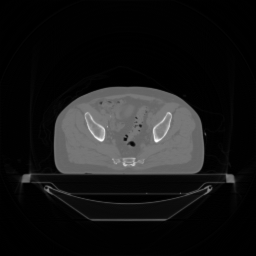} } 
	\subfloat[{\tt SARAH= 51.60}]{ \includegraphics[width=0.3\linewidth]{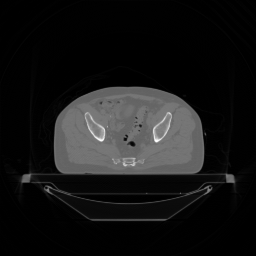} } 	\\[-2mm]	
    \caption{ Final reconstruction image of different methods for the 55th slice.}
    \label{3d3}
\end{figure}

\section{Conclusion}
\label{sec:conclude}
In this paper, we propose stochastic version ADMM which incorporates a class of variance reduction gradient estimators including SAGA and SARAH. 
Based on an MSE bound and KL property, we establish the convergence for the proposed algorithm. 
We demonstrate the performance of the proposed algorithm with SARAH, SAGA gradient estimators for three different problems: graphed-guided fused lasso, wavelet frame based 2D CT reconstruction and TV-$L_0$ for 3D CT reconstruction. 

\paragraph{Acknowledgement}
This work was supported by  NSFC (No.11771288) and  National key research and development program (No.2017YFB0202902). We thank the Student Innovation Center at Shanghai Jiao Tong University for providing us the computing services. Jingwei Liang was partly supported by Leverhulme Trust and Newton Trust.

\begin{small}
\bibliographystyle{plain}
\bibliography{bib}
\end{small}

\appendix

\section{Proofs of main theorems}
\label{app1}

\subsection{Properties of SAGA gradient approximation}


The MSE bound, geometric decay and convergence of the SAGA estimator have been established in \cite{DTLDS}, we present them here for the sake of convenience.

\begin{lemma}[MSE bound]\label{saga-mse}
The SAGA estimator satisfies
\begin{equation}\label{saga1}
\mathbb{E}_t \| \tnabla H(x_t) - \nabla H(x_t) \|^2 \leq \sfrac{1}{bn} \sum\nolimits_{i=1}^n \| \nabla H_i (x_t) - \nabla H_i (\varphi_t^i) \|^2,
\end{equation}
as well as
\begin{equation}\label{eqe1}
\mathbb{E}_t \| \tnabla H(x_t) - \nabla H(x_t) \| \leq \sfrac{1}{\sqrt{bn}} \sum\nolimits_{i=1}^n \| \nabla H_i (x_t) - \nabla H_i (\varphi_t^i) \|.
\end{equation}
\end{lemma}

\begin{proof}
According to the definition of SAGA operator \eqref{SAGA}, we have
\begin{equation}\nonumber
\begin{aligned}
\mathbb{E}_t \| \tnabla H(x_t) - \nabla H(x_t) \|^2
&= \mathbb{E}_t \| \sfrac{1}{b} \msum_{j \in J_t} (\nabla H_j (x_t) - \nabla H_j (\varphi_t^{j}) ) + \sfrac{1}{n} \msum_{i=1}^n \nabla H_i (\varphi_t^i)  - \nabla H(x_t)\|^2\\
&= \sfrac{1}{b^2} \mathbb{E}_t \sum\nolimits_{j \in J_t} \| \nabla H_j(x_t) - \nabla H_j (\varphi_t^j) \|^2\\
&= \sfrac{1}{bn} \sum\nolimits_{i=1}^n \| \nabla H_i (x_t) - \nabla H_i (\varphi_t^i) \|^2,
\end{aligned}
\end{equation} 
where the second equality is obtained by $\mathbb{E} ( \xi - \mathbb{E}\xi)^2 = \mathbb{E} \xi^2$. By the Jensen inequality, we can say that
\begin{equation}\nonumber
\begin{aligned}
\mathbb{E}_t \| \tnabla H(x_t) - \nabla H(x_t) \| &\leq \sqrt{\mathbb{E}_t \| \tnabla H(x_t) - \nabla H(x_t) \|^2 }\\
&\leq \sfrac{1}{\sqrt{bn}} \sqrt{\msum_{i=1}^n \| \nabla H_i (x_t) - \nabla H_i (\varphi_t^i) \|^2}\\
&\leq \sfrac{1}{\sqrt{bn}} \sum\nolimits_{i=1}^n \| \nabla H_i (x_t) - \nabla H_i (\varphi_t^i) \|.
\end{aligned}
\end{equation}
Hence, we get \eqref{saga1} and \eqref{eqe1}.
\end{proof}

\begin{lemma}[Geometric decay]\label{saga-geo}
For the SAGA estimator, we define 
\begin{subequations}\label{saga-upsilon-gamma}
\begin{align}
&\Upsilon_{t+1} = \sfrac{1}{bn} \sum\nolimits_{i=1}^n \| \nabla H_i (x_{t+1}) - \nabla H_i (\varphi_{t+1}^i) \|^2, \nonumber\\
&\Gamma_{t+1} = \sfrac{1}{\sqrt{bn}} \sum\nolimits_{i=1}^n \| \nabla H_i (x_{t+1}) - \nabla H_i (\varphi_{t+1}^i) \|. \nonumber
\end{align}
\end{subequations}
Then the sequence $\{ \Upsilon_t \}_{t \geq 1}$ decays geometrically:
\begin{equation}\nonumber
\mathbb{E}_t{\Upsilon_{t+1}} \leq (1 - \rho) \Upsilon_t + V_{\Upsilon} \mathbb{E}_t \| x_{t+1} - x_t \|^2,
\end{equation}
where $\rho=  \sfrac{b}{n}$ and $V_{\Upsilon} = \frac{nL^2}{bn-1}$.
\end{lemma}

\begin{proof}
Now, we show that $\mathbb{E} \Upsilon_{t+1}$ is decreasing at a geometric rate by applying the inequality $\| a - c \|^2 \leq (1 + \delta) \| a - b \|^2 + (1 + \delta^{-1}) \| b - c \|^2$ twice.
\begin{equation}\nonumber
\begin{aligned}
\mathbb{E}_t \Upsilon_{t+1} 
&= \sfrac{1}{bn} \mathbb{E}_t \big[ \msum_{i=1}^n \| \nabla H_i (x_{t+1}) - \nabla H_i (\varphi_{t+1}^i) \|^2 \big]\\
&\leq \sfrac{1}{bn} \sum_{i=1}^n \mathbb{E}_t  \| \nabla H_i (x_{t+1}) - \nabla H_i (x_t) + \nabla H_i(x_t) - \nabla H_i (\varphi_{t+1}^i) \|^2\\
&\leq \sfrac{1 + \delta^{-1}}{bn}  \mathbb{E}_t  \sum_{i=1}^n \| \nabla H_i (x_{t+1}) - \nabla H_i (x_t) \|^2 + \sfrac{1 + \delta}{bn}  \mathbb{E}_t  \sum_{i=1}^n \| \nabla H_i (x_{t}) - \nabla H_i (\varphi_{t+1}^i) \|^2 \\
&\leq \sfrac{1 + \delta^{-1}}{bn}  \mathbb{E}_t  \sum_{i=1}^n \| \nabla H_i (x_{t+1}) - \nabla H_i (x_t) \|^2 + \sfrac{(1 + \delta)(1 - \sfrac{b}{n})}{bn} \sum_{i=1}^n \| \nabla H_i (x_{t}) - \nabla H_i (\varphi_{t}^i) \|^2 \\
&\leq \sfrac{(1 + \delta)(1 - \sfrac{b}{n})}{bn} \sum_{i=1}^n \| \nabla H_i (x_{t}) - \nabla H_i (\varphi_{t}^i) \|^2  + \sfrac{(1 + \delta^{-1}) L^2}{b}  \mathbb{E}_t   \| x_{t+1} - x_t \|^2.\\
\end{aligned}
\end{equation}
Choosing $\delta  = bn -1$, we have $\frac{(1 + \delta)(1 - \frac{b}{n})}{bn} = 1 - \frac{b}{n} <1$. Then we can get the geometric decay of $\{ \Upsilon_{t} \}_{t \geq 1}$:
\begin{equation}\nonumber
\mathbb{E}_t{\Upsilon_{t+1}} \leq (1 - \rho) \Upsilon_t + V_{r} \mathbb{E}_t \| x_{t+1} - x_t \|^2,
\end{equation}
where $\rho = \frac{b}{n}$ and $V_{\Upsilon}= \frac{nL^2}{bn-1}$. This completes the proof.
\end{proof}

\begin{lemma}[Convergence of estimator]\label{saga-zero}
If $\lim_{t \to \infty} \mathbb{E} \| x_t - x_{t-1}\| = 0,$ it follows that $\mathbb{E} \Upsilon_t \to 0$ and $\mathbb{E} \Gamma_t \to 0$.
\end{lemma}

\begin{proof}
From the conclusion in Lemma \ref{saga-geo}, we can see that if $\mathbb{E}\| x_t - x_{t-1}\|^2 \to 0$, then so do $\Upsilon_t$ and $\Gamma_t$. First, we show that $\sum_{i=1}^n \mathbb{E} \| \nabla H_i (x_t) - H_i (\varphi_t^i) \|^2 \to 0.$\\
\begin{equation}\nonumber
\begin{aligned}
\sum_{i=1}^n \mathbb{E}\| \nabla H_i(x_t) - \nabla H_i (\varphi_t^i) \|^2 
&\leq L^2 \sum_{i=1}^n \mathbb{E}\| x_t - \varphi_t^i \|^2\\
&\leq L^2 n(1+\sfrac{2n}{b}) \mathbb{E} \| x_t - x_{t-1} \|^2 + (1+\sfrac{b}{2n}) \sum_{i=1}^n \mathbb{E} \| x_{t-1} - \varphi_t^i \|^2\\
&\leq L^2n(1+ \sfrac{2n}{b}) \mathbb{E} \| x_t - x_{t-1} \|^2 + (1+\sfrac{b}{2n}) (1 - \sfrac{b}{n}) \sum_{i=1}^n \mathbb{E} \| x_{t-1} - \varphi_{t-1}^i \|^2\\
&\leq L^2n(1+ \sfrac{2n}{b}) \mathbb{E} \| x_t - x_{t-1} \|^2 + (1 - \sfrac{b}{2n}) \sum_{i=1}^n \mathbb{E} \| x_{t-1} - \varphi_{t-1}^i \|^2\\
&\leq L^2n(1+ \sfrac{2n}{b}) \sum_{l=1}^t (1 - \sfrac{b}{2n})^{t-l} \mathbb{E} \| x_l - x_{l-1} \|^2,
\end{aligned}
\end{equation}
where the third inequality is obtained by using the definition of $\varphi_t$ in Definition \ref{SAGA}. Furthermore, because $\| x_t - x_{t-1} \|^2 \to 0,$ it is clear that the bound on the right goes to zero as $t \to \infty.$ The fact that $\Gamma_t \to 0$ follows similarly:
\begin{equation}\nonumber
\begin{aligned}
\sum_{i=1}^n \mathbb{E} \| \nabla H_i (x_t) - \nabla H_i (\varphi_t^i) \| &\leq L \sum\nolimits_{i=1}^n \mathbb{E} \| x_t - \varphi_t^i \| \\
&\leq n L \| x_t - x_{t-1} \| + \sum\nolimits_{i=1}^n \mathbb{E} \| x_{t-1} - \varphi_t^i \| \\
&\leq n L \| x_t - x_{t-1} \| + (1 - \sfrac{b}{n}) \sum\nolimits_{i=1}^n \mathbb{E} \| x_{t-1} - \varphi_{t-1}^i \| \\
&\leq nL \sum\nolimits_{l=1}^t (1 - \sfrac{b}{n})^{t-l} \mathbb{E} \| x_t - x_{t-1} \|.
\end{aligned}
\end{equation}
As $\| x_t - x_{t-1} \| \to 0,$ it follows that the bound on the right goes to zero as $t \to \infty$, so $\Gamma_t \to 0$. 
\end{proof}

\subsection{Properties of SARAH gradient approximation}

%

\begin{lemma}[MSE bound]\label{sarah-mse}
The SARAH gradient estimator satisfies
\begin{equation}\nonumber
\mathbb{E}_t \| \tnabla H(x_t) - \nabla H(x_t) \|^2 \leq \Pa{ 1 -\sfrac{1}{p} }  \| \tnabla H(x_{t-1}) - \nabla H(x_{t-1}) \|^2 + \sfrac{L^2}{b} \| x_t - x_{t-1} \|^2,
\end{equation}
as well as
\begin{equation}\nonumber
\mathbb{E}_t \| \tnabla H(x_t) - \nabla H(x_t) \| 
\leq \sqrt{1 - \tfrac{1}{p}}  \| \tnabla H(x_{t-1} )- \nabla H(x_{t-1}) \| + \sfrac{L}{\sqrt{b}} \| x_t - x_{t-1} \|.
\end{equation}
\end{lemma}

\begin{proof}
By using the definition of SARAH estimator in Definition \ref{sarah}, we have
\begin{equation}\nonumber
\begin{aligned}
\mathbb{E}_{t,p} \tnabla H(x_t) 
&= \mathbb{E}_{t,p} \big[ \sfrac{1}{b} \msum_{j\in J_t} \nabla H_j(x_t) - \nabla H_j(x_{t-1}) + \tnabla H(x_{t-1})  \big]\\
&= \nabla H(x_t) - \nabla H(x_{t-1}) + \tnabla H(x_{t-1}).
\end{aligned}
\end{equation}
In the following, we begin with a bound on $\mathbb{E}_{t,p} \| \tnabla H(x_t) - \nabla H(x_t) \|^2$.
\begin{equation}\nonumber
\begin{aligned}
&\mathbb{E}_{t,p} \| \tnabla H(x_t) - \nabla H(x_t) \|^2\\
&= \mathbb{E}_{t,p} \| \tnabla H(x_{t-1}) - \nabla H(x_{t-1}) + \nabla H(x_{t-1}) - \nabla H(x_t) + \tnabla H(x_t) - \tnabla H(x_{t-1}) \|^2\\
&=\| \tnabla H(x_{t-1}) - \nabla H(x_{t-1}) \|^2 + \| \nabla H(x_{t-1}) - \nabla H(x_t) \|^2 + \mathbb{E}_{t,p} \| \tnabla H(x_t) - \tnabla H(x_{t-1})\|^2\\
&\qquad + 2\langle \nabla H(x_{t-1}) - \tnabla H(x_{t-1}), \nabla H(x_t) - \nabla H(x_{t-1}) \rangle\\
&\qquad - 2 \langle \nabla H(x_{t-1}) - \tnabla H(x_{t-1}) , \mathbb{E}_{t,p} [\tnabla H(x_t) - \tnabla H(x_{t-1}) ] \rangle\\
&\qquad - 2\langle \nabla H(x_{t}) - \nabla H(x_{t-1}), \mathbb{E}_{t,p} [\tnabla H(x_t) - \tnabla H(x_{t-1}) ] \rangle.\\
\end{aligned}
\end{equation}
To simplify the inner product terms, we use the fact that
\begin{equation}\nonumber
\mathbb{E}_{t,p} [\tnabla H(x_t) - \tnabla H(x_{t-1}) ]  = \nabla H(x_t) - \nabla H(x_{t-1}).
\end{equation}
With this equality established, we see that the second inner product is equal to 
\begin{equation}\nonumber
2 \langle \nabla H(x_{t-1}) - \tnabla H(x_{t-1}) , \mathbb{E}_{t,p} [\tnabla H(x_t) - \tnabla H(x_{t-1}) ] \rangle = 2\langle \nabla H(x_{t-1}) - \tnabla H(x_{t-1}), \nabla H(x_t) - \nabla H(x_{t-1}) \rangle.
\end{equation}
The third inner product is equal to 
\begin{equation}\nonumber
2\langle \nabla H(x_{t}) - \nabla H(x_{t-1}), \mathbb{E}_{t,p} [\tnabla H(x_t) - \tnabla H(x_{t-1}) ] \rangle = 2 \| \nabla H(x_t) - \nabla H(x_{t-1}) \|^2.
\end{equation}
Consequently, we have
\begin{equation}\nonumber
\begin{aligned}
&\mathbb{E}_{t,p} \| \tnabla H(x_t) - \nabla H(x_t) \|^2\\
&\leq \| \tnabla H(x_{t-1}) - \nabla H(x_{t-1}) \|^2 - \| \nabla H(x_{t}) - \nabla H(x_{t-1}) \|^2 + \mathbb{E}_{t,p} \| \tnabla H(x_t) - \tnabla H(x_{t-1}) \|^2\\
&\leq \| \tnabla H(x_{t-1}) - \nabla H(x_{t-1}) \|^2  + \mathbb{E}_{t,p} \| \tnabla H(x_t) - \tnabla H(x_{t-1}) \|^2.\\
\end{aligned}
\end{equation}
We can bound the second term by computing the expectation,
\begin{equation}\nonumber
\begin{aligned}
\mathbb{E}_{t,p} \| \tnabla H(x_t) - \tnabla H(x_{t-1}) \|^2 
&= \mathbb{E}_{t,p} \| \sfrac{1}{b} \Pa{  \msum_{j \in J_t} \nabla H_j(x_t) - \nabla H_j(x_{t-1}) } \|^2\\
&\leq \sfrac{1}{b^2} \mathbb{E}_{t,p} \big[  \msum_{j \in J_t} \| \nabla H_j(x_t) - \nabla H_j(x_{t-1}) \|^2 \big]\\
&=\sfrac{1}{bn} \sum_{i=1}^n \| \nabla H_i (x_t) - \nabla H_i (x_{t-1}) \|^2.
\end{aligned}
\end{equation}
Then we have
\begin{equation}\nonumber
\begin{aligned}
\mathbb{E}_{t,p} \| \tnabla H(x_t) - \nabla H(x_t) \|^2
\leq \| \tnabla H(x_{t-1}) - \nabla H(x_{t-1}) \|^2 + \sfrac{1}{bn} \sum_{i=1}^n \| \nabla H_i (x_t) - \nabla H_i (x_{t-1}) \|^2.
\end{aligned}
\end{equation}
When the full gradient is computed, the MSE is equal to zero, so
\begin{equation}\nonumber
\begin{aligned}
\mathbb{E}_t \| \tnabla H(x_t) - \nabla H(x_t) \|^2
&\leq \Pa{ 1 - \sfrac{1}{p} } \Pa{ \| \tnabla H(x_{t-1} )- \nabla H(x_{t-1}) \|^2 + \sfrac{1}{bn} \msum_{i=1}^n \| \nabla H_i (x_t) - \nabla H_i (x_{t-1}) \|^2 } \\
&\leq \Pa{ 1 - \sfrac{1}{p} }  \| \tnabla H(x_{t-1} )- \nabla H(x_{t-1}) \|^2 + \sfrac{L^2}{b} \| x_t - x_{t-1} \|^2.
\end{aligned}
\end{equation}
Due to Jensen's inequality:
\begin{equation}\nonumber
\begin{aligned}
\mathbb{E}_t \| \tnabla H(x_t) - \nabla H(x_t) \| 
&\leq \sqrt{\mathbb{E}_t \| \tnabla H(x_t) - \nabla H(x_t) \|^2 }\\
&\leq \sqrt{ \pa{1 - \tfrac{1}{p}}  \| \tnabla H(x_{t-1} )- \nabla H(x_{t-1}) \|^2 + \sfrac{L^2}{b} \| x_t - x_{t-1} \|^2 }\\
&\leq \sqrt{1 - \tfrac{1}{p}}  \| \tnabla H(x_{t-1} )- \nabla H(x_{t-1}) \| + \sfrac{L}{\sqrt{b}} \| x_t - x_{t-1} \|.
\end{aligned}
\end{equation}
This completes the proof. 
\end{proof}

\begin{lemma}[Geometric decay]\label{sarah-geo}
For the SARAH gradient estimator, we define 
\begin{subequations}\nonumber
\begin{align}
&\Upsilon_{t+1} =  \| \tnabla H (x_{t}) - \nabla H (x_{t}) \|^2, \label{sarah-upsilon}\\
&\Gamma_{t+1} = \| \tnabla H (x_{t}) - \nabla H (x_{t}) \|. \label{sarah-gamma}
\end{align}
\end{subequations}
Then the sequence $\{ \Upsilon_t \}_{t \geq 1}$ decays geometrically:
\begin{equation}\label{sarah12}
\mathbb{E}_t{\Upsilon_{t+1}} \leq (1 - \rho) \Upsilon_t + V_{\Upsilon} \mathbb{E}_t \| x_{t} - x_{t-1} \|^2,
\end{equation}
where $\rho= \sfrac{1}{p}$ and $V_r = \sfrac{L^2}{b}$.
\end{lemma}

\begin{proof}
{This is a direct result of Lemma \ref{sarah-mse}.}
\end{proof}

\begin{lemma}[Convergence of estimator]\label{sarah-zero}
If $\lim_{t \to \infty} \mathbb{E} \| x_t - x_{t-1}\| = 0,$ it follows that $\mathbb{E} \Upsilon_t \to 0$ and $\mathbb{E} \Gamma_t \to 0$.
\end{lemma}

\begin{proof}
By \eqref{sarah12} we have
\begin{equation}\nonumber
\begin{aligned}
\mathbb{E} \Upsilon_t 
\leq (1 - \sfrac{1}{p}) \mathbb{E} \Upsilon_{t-1} + \sfrac{L^2}{b} \mathbb{E}  \| x_t - x_{t-1} \|^2 
\leq \sfrac{L^2}{b} \sum_{l=1}^t (1 - \sfrac{1}{p})^{t-l} \mathbb{E} \| x_l - x_{l-1} \|^2.
\end{aligned}
\end{equation}
It implies $\mathbb{E} \Upsilon_t \to 0.$  By Jensen's inequality, we have $\mathbb{E} \Gamma_t \to 0$ as well.
\end{proof}

\subsection{Proofs of main theorems}

Before giving the proof of Theorem \ref{decrease}, we first show that the sequence $\{ \Psi_t \}$ is bounded from below in the following lemma, which will be used in the proof of Theorem \ref{decrease}. Recall
\begin{equation}\label{not}
\begin{aligned}
&\Phi(x, z, u, x^{\prime}, u^{\prime}) \eqdef L_{\beta} (x, z, u) + \sfrac{4(1- \sigma)}{\sigma^2\beta\lambda_{m}} \| A^T ( u- u^{\prime} ) + \sigma B(x - x^{\prime})  \|^2 + \sfrac{8 (\sigma \tau + L)^2}{\sigma\beta\lambda_{m}} \| x - x^{\prime} \|^2\\
&\Phi_t \eqdef \Phi(x_t, z_t, u_t, x_{t-1}, u_{t-1}), ~~~~X_t = (x_t, z_t, u_t, x_{t-1}, u_{t-1}), ~~~~X^* = (x^*,z^*, u^*, x^*, u^*).
\end{aligned}
\end{equation}
{We also denote $\underline{\Phi} \eqdef \inf_{(x,z)} \{ F(z) + H(x) \} > -\infty.$}
\begin{lemma}\label{phi-bound}
Suppose that (a1) of Assumption \ref{assum} is satisfied and $\{x_t, z_t, u_t \}_{t \geq 0}$ is a sequence generated by Algorithm \ref{alg:SADMM}. Then the sequence $\{ \Psi_t \}_{t \geq 0}$ is bounded from below.
\end{lemma}


\begin{proof}
We show that $\underline{\Phi}$ is a lower bound of $\{ \Psi_t \}_{t \geq 1}$. Suppose by contradiction that there exists $t_0 \geq 1$ such that $\Psi_{t_0} - \underline{\Phi} < 0.$ According to \eqref{dec-relation}, for any $T \geq t_0$,
$$
\sum_{t = 1}^{T} \mathbb{E}(\Psi_t - \underline{\Phi}) \leq \sum_{t=1}^{t_0 - 1} \mathbb{E}(\Psi_t - \Phi) + (T - t_0 +1) (\Psi_{t_0} - \underline{\Phi}),
$$
which implies that
$$
\lim_{T \to \infty} \sum_{t= 1}^{T} \mathbb{E}(\Psi_t - \underline{\Phi}) = -\infty.
$$
On the other hand, for any $t \geq 1$ it holds that
\begin{equation}\nonumber
\begin{aligned}
\Psi_t - \underline{\Phi}&\geq F(z_t) + H(x_t) + \langle u_t, Ax_t - z_t \rangle - \underline{\Phi}\\
&\geq \langle u_t, Ax_t - z_t \rangle\\
&=\sfrac{1}{\sigma\beta} \langle u_t, u_t - u_{t-1} \rangle\\
&=\sfrac{1}{2\sigma\beta} \| u_t \|^2 + \sfrac{1}{2\sigma\beta} \| u_t - u_{t-1} \|^2 - \sfrac{1}{2\sigma\beta} \| u_{t-1} \|^2.
\end{aligned}
\end{equation}
Therefore, for any $T \geq 1$, we have
$$
\sum_{t= 1}^{T} \mathbb{E} (\Psi_t - \underline{\Phi}) \geq \sfrac{1}{2\sigma\beta}\sum_{t=1}^T \mathbb{E}\| u_t - u_{t-1} \|^2 + \sfrac{1}{2\sigma\beta} \| u_T\|^2 -  \sfrac{1}{2\sigma\beta} \| u_0\|^2 \geq  -\sfrac{1}{2\sigma\beta} \| u_0\|^2,
$$
which leads to a contradiction.
\end{proof}

\begin{proof}[Proof of Theorem \ref{decrease}]
From \eqref{sub1} we obtain
\begin{equation}\label{eq1}
F(z_{t+1}) + \langle u_t, Ax_t - z_{t+1} \rangle + \sfrac{\beta}{2} \| Ax_t - z_{t+1} \|^2 \leq F(z_t) + \langle u_t, Ax_t - z_t \rangle + \sfrac{\beta}{2} \| Ax_t - z_t \|^2.
\end{equation}
On the other hand, as the gradient of $H$ is Lipschitz continuous, we have
\begin{equation}\nonumber
H(x_{t+1}) \leq H(x_t) + \langle \nabla H(x_t), x_{t+1} - x_t \rangle + \sfrac{L}{2} \| x_{t+1} - x_t \|^2.
\end{equation}
By taking into consideration \eqref{sub2}, we get
\begin{equation}\label{eq2}
\begin{aligned}
H(x_{t+1}) 
&\leq H(x_t) + \langle \tnabla H(x_t), x_{t+1} - x_t \rangle + \sfrac{L}{2} \| x_{t+1} - x_t \|^2 +  \langle \nabla H(x_t) - \tnabla H(x_t), x_{t+1} - x_t \rangle \\
&\leq H(x_t) + \langle \tau(x_t - x_{t+1}) - A^{T} u_t + \beta A^{T}(Ax_t - z_{t+1}), x_{t+1} - x_t \rangle\\
&\qquad +\sfrac{L}{2} \| x_{t+1} - x_t \|^2 + \langle \nabla H(x_t) - \tnabla H(x_t), x_{t+1} - x_t \rangle\\
&\leq H(x_t) - \tau \| x_{t+1} - x_t \|^2 - \langle u_t, A(x_{t+1} - x_t) \rangle - \beta\langle Ax_t- z_{t+1}, Ax_{t+1} - Ax_t \rangle\\
&\qquad+\sfrac{L}{2} \| x_{t+1} - x_t \|^2 + \langle \nabla H(x_t) - \tnabla H(x_t), x_{t+1} - x_t \rangle\\
&\leq H(x_t) - \langle u_t, Ax_{t+1} - Ax_t \rangle +\sfrac{\beta}{2} \| Ax_t - z_{t+1} \|^2 - \sfrac{\beta}{2} \| Ax_{t+1} - z_{t+1} \|^2\\
&\qquad - (\tau - \sfrac{L +\beta \| A \|^2}{2}) \| x_{t+1} - x_t \|^2 + \langle \nabla H(x_t) - \tnabla H(x_t), x_{t+1} - x_t \rangle,\\
\end{aligned}
\end{equation}
where the last inequality is due to $\langle a, b \rangle = \sfrac{1}{2} ( \| a + b \|^2 - \| a \|^2 - \| b \|^2 )$. Rewrite \eqref{eq2}, we get
\begin{equation}\label{eq3}
\begin{aligned}
H(x_t) 
&\geq H(x_{t+1}) + \langle u_t, Ax_{t+1} - Ax_t \rangle - \sfrac{\beta}{2} \| Ax_t - z_{t+1} \|^2 + \sfrac{\beta}{2} \| Ax_{t+1} - z_{t+1} \|^2\\
&\qquad + (\tau - \sfrac{L +\beta \| A \|^2}{2}) \| x_{t+1} - x_t \|^2 - \langle \nabla H(x_t) - \tnabla H(x_t), x_{t+1} - x_t \rangle .
\end{aligned}
\end{equation}
Combining \eqref{eq1} and \eqref{eq3}, we get
\begin{equation}\nonumber
\begin{aligned}
&F(z_{t+1}) + H(x_{t+1}) + \langle u_t, Ax_{t+1} - z_{t+1} \rangle + \sfrac{\beta}{2} \| Ax_{t+1} - z_{t+1} \|^2\\
&\qquad + (\tau - \sfrac{L +\beta \| A \|^2}{2}) \| x_{t+1} - x_t \|^2 - \langle \nabla H(x_t) - \tnabla H(x_t), x_{t+1} - x_t \rangle\\
&\leq F(z_t) + H(x_t) + \langle u_t, Ax_t - z_t \rangle + \sfrac{\beta}{2} \| Ax_t - z_t \|^2.
\end{aligned}
\end{equation}
From the equation \eqref{sub3}, we have
\begin{equation}\nonumber
\begin{aligned}
&F(z_{t+1}) + H(x_{t+1}) + \langle u_{t+1}, Ax_{t+1} - z_{t+1} \rangle - \sigma\beta \| Ax_{t+1} - z_{t+1} \|^2 + \sfrac{\beta}{2} \| Ax_{t+1} - z_{t+1} \|^2\\
&\qquad + (\tau - \sfrac{L +\beta \| A \|^2}{2}) \| x_{t+1} - x_t \|^2 - \langle \nabla H(x_t) - \tnabla H(x_t), x_{t+1} - x_t \rangle\\
&\leq F(z_t) + H(x_t) + \langle u_t, Ax_t - z_t \rangle + \sfrac{\beta}{2} \| Ax_t - z_t \|^2.
\end{aligned}
\end{equation}
By adding $\sfrac{2}{\sigma\beta} \| u_{t+1} - u_t \|^2$, we have
\begin{equation}\label{eq7}
\begin{aligned}
&F(z_{t+1}) + H(x_{t+1}) + \langle u_{t+1}, Ax_{t+1} - z_{t+1} \rangle + \sfrac{\beta}{2} \| Ax_{t+1} - z_{t+1} \|^2\\
&\qquad+ (\tau - \sfrac{L +\beta \| A \|^2}{2}) \| x_{t+1} - x_t \|^2 + \sfrac{1}{\sigma\beta} \| u_{t+1} - u_{t} \|^2 - \langle \nabla H(x_t) - \tnabla H(x_t), x_{t+1} - x_t \rangle\\
&\leq F(z_t) + H(x_t) + \langle u_t, Ax_t - z_t \rangle + \sfrac{\beta}{2} \| Ax_t - z_t \|^2 + \sfrac{2}{\sigma\beta} \| u_{t+1} - u_{t} \|^2 .
\end{aligned}
\end{equation}
Next, we will focus on estimating $\| u_{t+1} - u_t \|^2$. We can rewrite \eqref{sub2} as 
\begin{equation}\nonumber
\begin{aligned}
\tau (x_t - x_{t+1})&=\tnabla  H(x_t) + A^{T} u_t + \beta A^{T} (Ax_{t+1} - z_{t+1}) + \beta A^T A(x_t - x_{t+1})\\
&= \tnabla  H(x_t) + A^{T} u_t + \sfrac{1}{\sigma} A^{T} (u_{t+1} - u_{t}) + \beta A^T A(x_t - x_{t+1})\\
\end{aligned}
\end{equation}
where the last equation is due to \eqref{sub3}. After multiplying both sides by $\sigma$ and rearranging the terms, we get
\[
A^T u_{t+1} + \sigma B(x_{t+1} - x_t) = (1 - \sigma) A^T u_t - \sigma \tnabla H(x_t) ,
\]
which is equivalent to 
\begin{equation}\nonumber
A^T u_{t+1} + \sigma B(x_{t+1} - x_t) = (1 - \sigma) A^T u_t - \sigma \nabla H(x_t) - \sigma \tnabla H(x_t) + \sigma \nabla H(x_t).
\end{equation}
Since $t$ is arbitrary, we also have
\begin{equation}\nonumber
A^T u_{t} + \sigma B(x_{t} - x_{t-1}) = (1 - \sigma) A^T u_{t-1} - \sigma \nabla H(x_{t-1}) - \sigma \tnabla H(x_{t-1}) + \sigma \nabla H(x_{t-1}).
\end{equation}
Subtracting these two identities yields
\begin{equation}\nonumber
\begin{aligned}
A^T( u_{t+1}  - u_t)+ \sigma B(x_{t+1} - x_t) = &(1 - \sigma)  A^T ( u_t - u_{t-1} ) + \sigma B(x_t - x_{t-1}) + \sigma \Pa{ \nabla H(x_{t-1}) - \nabla H(x_{t}) }\\
& - \sigma \Pa{ \tnabla  H(x_t) - \nabla H(x_{t})} +\sigma \Pa{ \tnabla  H(x_{t-1}) - \nabla H(x_{t-1}) }.
\end{aligned}
\end{equation}
Notice that $0 < \sigma < 1$ and $\| B \| \leq \tau$, by Cauchy inequality we have
\begin{equation}\label{eq14}
\begin{aligned}
&\| A^T( u_{t+1}  - u_t)+ \sigma B(x_{t+1} - x_t) \|^2 \leq (1 - \sigma) \| A^T ( u_t - u_{t-1} ) + \sigma B(x_t - x_{t-1}) \|^2\\
&\qquad + \sigma \Pa{ \| \sigma B(x_t - x_{t-1}) \| + \| \Pa{\nabla H(x_{t-1}) - \nabla H(x_{t})}  - \Pa{\tnabla  H(x_t) - \nabla H(x_{t})} + \Pa{ \tnabla  H(x_{t-1}) - \nabla H(x_{t-1}) }\| }^2\\
&\leq (1 - \sigma) \| A^T ( u_t - u_{t-1} ) + \sigma B(x_t - x_{t-1}) \|^2 + 2\sigma(\sigma\tau +L)^2 \| x_t - x_{t-1} \|^2 \\
&\qquad + 2 \sigma \| \Pa{ \tnabla  H(x_t) - \nabla H(x_{t})} - \Pa{ \tnabla  H(x_{t-1}) - \nabla H(x_{t-1}) } \|^2.
\end{aligned}
\end{equation}
On the other hand, since the linear operator $A$ is surjective, we have
$$
\sfrac{\lambda_{m}}{2} \| u_{t+1} - u_t \|^2 \leq \sfrac{1}{2} \| A^T (u_{t+1} - u_t) \|^2 \leq \| A^T( u_{t+1}  - u_t)+ \sigma B(x_{t+1} - x_t)  \|^2 + \sigma^2\tau^2 \| x_{t+1} - x_t \|^2.\\
$$
After combining this with \eqref{eq14}, we get
\begin{equation}\nonumber
\begin{aligned}
&\sfrac{\sigma\lambda_{m}}{2} \| u_{t+1} - u_t \|^2 + (1 - \sigma) \| A^T( u_{t+1}  - u_t)+ \sigma B(x_{t+1} - x_t)  \|^2 - \sigma^3\tau^2 \| x_{t+1} - x_t \|^2\\
&\leq (1- \sigma) \| A^T ( u_t - u_{t-1} ) + \sigma B(x_t - x_{t-1}) \|^2 +  2\sigma (\sigma \tau + L)^2 \| x_t - x_{t-1} \|^2\\
&\qquad +2 \sigma \| \Pa{ \tnabla  H(x_t) - \nabla H(x_{t}) } - \Pa{ \tnabla  H(x_{t-1}) - \nabla H(x_{t-1}) }  \|^2.
\end{aligned}
\end{equation}
Multiplying the above relation by $\sfrac{4}{\sigma^2 \beta \lambda_{m}} > 0$, we have
\begin{equation}\nonumber
\begin{aligned}
&\sfrac{2}{\sigma\beta} \| u_{t+1} - u_t \|^2 - \sfrac{4 \sigma \tau^2}{\beta \lambda_{m}} \| x_{t+1} - x_t \|^2 + \sfrac{4(1- \sigma)}{\sigma^2\beta\lambda_{m}} \| A^T( u_{t+1}  - u_t)+ \sigma B(x_{t+1} - x_t) \|^2\\
&\leq \sfrac{4(1- \sigma)}{\sigma^2\beta\lambda_{m}} \| A^T ( u_t - u_{t-1} ) + \sigma B(x_t - x_{t-1})  \|^2 + \sfrac{8 (\sigma \tau + L )^2}{\sigma\beta\lambda_{m}} \| x_t - x_{t-1} \|^2 \\
&\qquad + \sfrac{8}{\sigma\beta\lambda_{m}} \| \Pa{ \tnabla  H(x_t) - \nabla H(x_{t}) } - \Pa{ \tnabla  H(x_{t-1}) - \nabla H(x_{t-1}) }  \|^2.
\end{aligned}
\end{equation}
Adding the resulting inequality to \eqref{eq7} yields
\begin{equation}\nonumber
\begin{aligned}
&F(z_{t+1}) + H(x_{t+1}) + \langle u_{t+1}, Ax_{t+1} - z_{t+1} \rangle + \sfrac{\beta}{2} \| Ax_{t+1} - z_{t+1} \|^2\\
&\qquad + \sfrac{4(1- \sigma)}{\sigma^2\beta\lambda_{m}} \| A^T( u_{t+1}  - u_t)+ \sigma B(x_{t+1} - x_t) \|^2 + ( \tau - \sfrac{L + \beta \| A \|^2}{2}) \| x_{t+1} - x_t \|^2\\
&\qquad + \sfrac{1}{\sigma\beta} \| u_{t+1} - u_t \|^2- \langle \nabla H(x_t) - \tnabla  H(x_t), x_{t+1} - x_t \rangle  - \sfrac{4 \sigma \tau^2}{\beta \lambda_{m}} \| x_{t+1} - x_t \|^2 \\
&\leq F(z_t) + H(x_t) + \langle u_t, Ax_t - z_t \rangle + \sfrac{\beta}{2} \| Ax_{t} - z_{t} \|^2+ \sfrac{4(1- \sigma)}{\sigma^2\beta\lambda_{m}} \| A^T ( u_t - u_{t-1} ) + \sigma B(x_t - x_{t-1})  \|^2 \\
&\qquad + \sfrac{8 (\sigma \tau + L)^2}{\sigma\beta\lambda_{m}} \| x_{t} - x_{t-1} \|^2 + \sfrac{8}{\sigma\beta\lambda_{m}} \| \Pa{ \tnabla  H(x_t) - \nabla H(x_{t}) } - \Pa{ \tnabla  H(x_{t-1}) - \nabla H(x_{t-1}) } \|^2.
\end{aligned}
\end{equation}
It follows from Cauchy inequality that
\begin{equation}\nonumber
\begin{aligned}
&F(z_{t+1}) + H(x_{t+1}) + \langle u_{t+1}, Ax_{t+1} - z_{t+1} \rangle + \sfrac{\beta}{2} \| Ax_{t+1} - z_{t+1} \|^2 \\
&\qquad + \sfrac{4(1- \sigma)}{\sigma^2\beta\lambda_{m}} \| A^T( u_{t+1}  - u_t)+ \sigma B(x_{t+1} - x_t) \|^2 +  \sfrac{8 (\sigma \tau + L)^2}{\sigma\beta\lambda_{m}} \| x_{t+1} - x_{t} \|^2\\
&\qquad + ( \tau - \sfrac{L + \beta \| A \|^2}{2} - \sfrac{4 \sigma \tau^2}{\beta \lambda_{m}} -  \sfrac{8 (\sigma \tau + L)^2}{\sigma\beta\lambda_{m}}) \| x_{t+1} - x_t \|^2 + \sfrac{1}{\sigma\beta} \| u_{t+1} - u_t \|^2\\
&\leq F(z_t) + H(x_t) + \langle u_t, Ax_t - z_t \rangle + \sfrac{\beta}{2} \| Ax_{t} - z_{t} \|^2 + \sfrac{4(1- \sigma)}{\sigma^2\beta\lambda_{m}} \| A^T ( u_t - u_{t-1} ) + \sigma B(x_t - x_{t-1})  \|^2\\
&\qquad + \sfrac{8 (\sigma \tau + L)^2}{\sigma\beta\lambda_{m}} \| x_{t} - x_{t-1} \|^2 + \sfrac{8}{\sigma\beta\lambda_{m}} \| \Pa{ \tnabla  H(x_t) - \nabla H(x_{t}) } - \Pa{ \tnabla  H(x_{t-1}) - \nabla H(x_{t-1}) }  \|^2\\
&\qquad + \langle \nabla H(x_t) - \tnabla  H(x_t), x_{t+1} - x_t \rangle\\
&\leq F(z_t) + H(x_t) + \langle u_t, Ax_t - z_t \rangle + \sfrac{\beta}{2} \| Ax_{t} - z_{t} \|^2+ \sfrac{4(1- \sigma)}{\sigma^2\beta\lambda_{m}} \| A^T ( u_t - u_{t-1} ) + \sigma B(x_t - x_{t-1})  \|^2\\
&\qquad + \sfrac{8 (\sigma \tau + L)^2}{\sigma\beta\lambda_{m}} \| x_{t} - x_{t-1} \|^2 + \sfrac{16}{\sigma\beta\lambda_{m}} \| \tnabla  H(x_t) - \nabla H(x_{t}) \|^2 \\
&\qquad + \sfrac{16}{\sigma\beta\lambda_{m}}  \| \tnabla  H(x_{t-1}) - \nabla H(x_{t-1}) \|^2 + \sfrac{\eta}{2} \| \nabla H(x_t) - \tnabla  H(x_t)\|^2 + \sfrac{1}{2\eta} \| x_{t+1} - x_t \|^2,
\end{aligned}
\end{equation}
where $\eta > 0$ is a constant. Applying the conditional expectation operator $\mathbb{E}_t$, 
\begin{equation}\nonumber
\begin{aligned}
&\mathbb{E}_t \left[ \Phi_{t+1} + ( \tau - \sfrac{L + \beta \| A \|^2}{2} - \sfrac{4\sigma\tau^2}{\beta\lambda_{m}} - \sfrac{8(\sigma \tau + L)^2}{\sigma\beta\lambda_{m}} )\| x_{t+1} - x_t \|^2 + \sfrac{1}{\sigma\beta} \| u_{t+1} - u_t \|^2 \right]\\
&\leq \Phi_t + ( \sfrac{16}{\sigma\beta\lambda_{m}} + \sfrac{\eta}{2} ) \mathbb{E}_t \| \tnabla  H(x_t) - \nabla H(x_t) \|^2 +  \sfrac{16}{\sigma\beta\lambda_{m}}  \mathbb{E}_t \| \tnabla  H(x_{t-1}) - \nabla H(x_{t-1}) \|^2 \\
&\qquad + \sfrac{1}{2\eta} \mathbb{E}_t \| x_{t+1} - x_t \|^2.
\end{aligned}
\end{equation}
Adding $\sfrac{16}{\sigma\beta\lambda_{m}}  \mathbb{E}_t \| \tnabla  H(x_{t}) - \nabla H(x_{t}) \|^2 $ on the both sides,
\begin{equation}\label{ad1}
\begin{aligned}
&\mathbb{E}_t \big[ \Phi_{t+1} + \sfrac{16}{\sigma\beta\lambda_{m}}  \| \tnabla  H(x_{t}) - \nabla H(x_{t}) \|^2+ ( \tau - \sfrac{L + \beta \| A \|^2}{2} - \sfrac{4\sigma\tau^2}{\beta\lambda_{m}} - \sfrac{8(\sigma \tau + L)^2}{\sigma\beta\lambda_{m}})\| x_{t+1} - x_t \|^2 \\
&\qquad+ \sfrac{1}{\sigma\beta} \| u_{t+1} - u_t \|^2 \big]\\
&\leq \Phi_t +  \sfrac{16}{\sigma\beta\lambda_{m}}  \| \tnabla  H(x_{t-1}) - \nabla H(x_{t-1}) \|^2 + ( \sfrac{32}{\sigma\beta\lambda_{m}} + \sfrac{\eta}{2} ) \mathbb{E}_t \| \tnabla  H(x_t) - \nabla H(x_t) \|^2 + \sfrac{1}{2\eta} \mathbb{E}_t \| x_{t+1} - x_t \|^2.
\end{aligned}
\end{equation}
Simplifying the above inequality, 
\begin{equation}\nonumber
\begin{aligned}
&\mathbb{E}_t \big[  \Phi_{t+1} + \sfrac{16}{\sigma\beta\lambda_{m}}  \| \tnabla  H(x_{t}) - \nabla H(x_{t}) \|^2+ ( \tau - \sfrac{L + \beta \| A \|^2}{2} - \sfrac{4\sigma\tau^2}{\beta\lambda_{m}} - \sfrac{8(\sigma \tau + L)^2}{\sigma\beta\lambda_{m}} -  \sfrac{1}{2\eta} ) \| x_{t+1} - x_t \|^2 \\
&\qquad + \sfrac{1}{\sigma\beta} \| u_{t+1} - u_t \|^2 \big]\\
&\leq  \Phi_t +  \sfrac{16}{\sigma\beta\lambda_{m}}  \| \tnabla  H(x_{t-1}) - \nabla H(x_{t-1}) \|^2+ ( \sfrac{32}{\sigma\beta\lambda_{m}} + \sfrac{\eta}{2} ) \mathbb{E}_t \| \tnabla  H(x_t) - \nabla H(x_t) \|^2\\
&\leq  \Phi_t +  \sfrac{16}{\sigma\beta\lambda_{m}}  \| \tnabla  H(x_{t-1}) - \nabla H(x_{t-1}) \|^2+ ( \sfrac{32}{\sigma\beta\lambda_{m}} + \sfrac{\eta}{2} ) \Upsilon_{t} + ( \sfrac{32}{\sigma\beta\lambda_{m}} + \sfrac{\eta}{2} ) V_1 (\mathbb{E}_t \| x_{t+1} - x_t \|^2 + \| x_t - x_{t-1} \|^2).\\
\end{aligned}
\end{equation}
Then we have
\begin{equation}\label{add1}
\begin{aligned}
&\mathbb{E}_t \big[ \Phi_{t+1} + \sfrac{16}{\sigma\beta\lambda_{m}}  \| \tnabla  H(x_{t}) - \nabla H(x_{t}) \|^2 + \sfrac{1}{\sigma\beta} \| u_{t+1} - u_t \|^2 \\
&\qquad + \Pa{ \tau - \sfrac{L + \beta \| A \|^2}{2} - \sfrac{4\sigma\tau^2}{\beta\lambda_{m}} - \sfrac{8(\sigma \tau + L)^2}{\sigma\beta\lambda_{m}} -  \sfrac{1}{2\eta} - ( \sfrac{32}{\sigma\beta\lambda_{m}} + \sfrac{\eta}{2} ) V_1 } \| x_{t+1} - x_t \|^2 \big]\\
&\leq \Phi_t + \sfrac{16}{\sigma\beta\lambda_{m}}  \| \tnabla  H(x_{t-1}) - \nabla H(x_{t-1}) \|^2 + ( \sfrac{32}{\sigma\beta\lambda_{m}} + \sfrac{\eta}{2} ) \Upsilon_{t} + ( \sfrac{32}{\sigma\beta\lambda_{m}} + \sfrac{\eta}{2} ) V_1 \| x_t - x_{t-1} \|^2.
\end{aligned}
\end{equation}
Next using \eqref{var-geo}, we have
\begin{equation}\label{add2}
\begin{aligned}
( \sfrac{32}{\sigma\beta\lambda_{m}} + \sfrac{\eta}{2} ) \Upsilon_{t}
&\leq ( \sfrac{32}{\sigma\beta\lambda_{m}} + \sfrac{\eta}{2} ) \sfrac{1}{\rho} \bPa{ -\mathbb{E}_t \Upsilon_{t+1} + \Upsilon_t  + V_{\Upsilon} ( \mathbb{E}_t \| x_{t+1} - x_t \|^2 + \| x_t - x_{t-1} \|^2 ) }.
\end{aligned}
\end{equation}
Combining \eqref{add1} and \eqref{add2}, we have
\begin{equation}\nonumber
\begin{aligned}
&\mathbb{E}_t \big[ \Phi_{t+1} + \sfrac{16}{\sigma\beta\lambda_{m}}  \| \tnabla  H(x_{t}) - \nabla H(x_{t}) \|^2 + \sfrac{1}{\rho} ( \sfrac{32}{\sigma\beta\lambda_{m}} + \sfrac{\eta}{2} ) \Upsilon_{t+1} + \sfrac{1}{\sigma\beta} \| u_{t+1} - u_t \|^2 \\
&\qquad + \Pa{ \tau - \sfrac{L + \beta \| A \|^2}{2} - \sfrac{4\sigma\tau^2}{\beta\lambda_{m}} - \sfrac{8(\sigma \tau + L)^2}{\sigma\beta\lambda_{m}} -  \sfrac{1}{2\eta} - ( \sfrac{32}{\sigma\beta\lambda_{m}} + \sfrac{\eta}{2} ) (V_1 + \sfrac{V_{\Upsilon}}{\rho}) } \| x_{t+1} - x_t \|^2 \big]\\
&\leq \Phi_t + \sfrac{16}{\sigma\beta\lambda_{m}}  \| \tnabla  H(x_{t-1}) - \nabla H(x_{t-1}) \|^2 + \sfrac{1}{\rho}( \sfrac{32}{\sigma\beta\lambda_{m}} + \sfrac{\eta}{2} ) \Upsilon_{t} + ( \sfrac{32}{\sigma\beta\lambda_{m}} + \sfrac{\eta}{2} ) (V_1 + \sfrac{V_{\Upsilon}}{\rho})  \| x_{t} - x_{t-1} \|^2.\\
\end{aligned}
\end{equation}
This is eqivalent to 
\begin{equation}\nonumber
\begin{aligned}
&\mathbb{E}_t \big[ \Phi_{t+1} + \sfrac{16}{\sigma\beta\lambda_{m}}  \| \tnabla  H(x_{t}) - \nabla H(x_{t}) \|^2 + \sfrac{1}{\rho} ( \sfrac{32}{\sigma\beta\lambda_{m}} + \sfrac{\eta}{2} ) \Upsilon_{t+1}\\
&\qquad + \Pa{  ( \sfrac{32}{\sigma\beta\lambda_{m}} + \sfrac{\eta}{2} ) (V_1 + \sfrac{V_{\Upsilon}}{\rho}) + C_2 } \| x_{t+1} - x_t \|^2  + \sfrac{1}{\sigma\beta} \| u_{t+1} - u_t \|^2 + C_2 \| x_t - x_{t-1} \|^2\\
&\qquad + \Pa{ \tau - \sfrac{L + \beta \| A \|^2}{2} - \sfrac{4\sigma\tau^2}{\beta\lambda_{m}} - \sfrac{8(\sigma \tau + L)^2}{\sigma\beta\lambda_{m}} -  \sfrac{1}{2\eta} - (\sfrac{64}{\sigma\beta\lambda_{m}} + \eta ) (V_1 + \sfrac{V_{\Upsilon}}{\rho}) - C_2 } \| x_{t+1} - x_t \|^2 \big]\\
&\leq \Phi_t + \sfrac{16}{\sigma\beta\lambda_{m}}  \| \tnabla  H(x_{t-1}) - \nabla H(x_{t-1}) \|^2 + \sfrac{1}{\rho} ( \sfrac{32}{\sigma\beta\lambda_{m}} + \sfrac{\eta}{2} ) \Upsilon_{t}\\
&\qquad + \Pa{ ( \sfrac{32}{\sigma\beta\lambda_{m}} + \sfrac{\eta}{2} ) (V_1 + \sfrac{V_{\Upsilon}}{\rho}) + C_2 } \| x_{t} - x_{t-1} \|^2\\
\end{aligned}
\end{equation}
for some constant $C_2 \geq 0$. Then we have
\begin{equation}\label{eq31}
\mathbb{E}_t [ \Psi_{t+1} + \tilde{\eta}  \| x_{t+1} - x_t \|^2 + \sfrac{1}{\sigma\beta} \| u_{t+1} - u_t \|^2  + C_2 \| x_t - x_{t-1} \|^2 ] \leq \Psi_t.
\end{equation}
Thus \eqref{dec-relation} is  established.

Now we apply the full expectation operator to \eqref{eq31} and sum the resulting inequality from $t = 0$ to $t = T -1$, 
\begin{equation}\nonumber
\mathbb{E} \Psi_{T} + \tilde{\eta}  \sum_{t=0}^{T-1} \mathbb{E} \| x_{t+1} - x_t \|^2 + \sfrac{1}{\sigma\beta} \sum_{t=0}^{T-1} \mathbb{E} \| u_{t+1} - u_t \|^2 + C_2 \sum_{t=0}^{T-1} \mathbb{E} \| x_{t} - x_{t -1} \|^2 \leq \Psi_0.
\end{equation}
From Lemma \ref{phi-bound}, we know
\begin{equation}\nonumber
\tilde{\eta}  \sum_{t=0}^{T-1} \mathbb{E} \| x_{t+1} - x_t \|^2 + \sfrac{1}{\sigma\beta} \sum_{t=0}^{T-1} \mathbb{E} \| u_{t+1} - u_t \|^2 + C_2 \sum_{t=0}^{T-1} \mathbb{E} \| x_{t} - x_{t -1} \|^2 \leq \Psi_0 - \underline{\Phi}.
\end{equation}
Taking the limit $T \to \infty$, we have $\mathbb{E} \| x_{t+1} - x_t \| \to 0$ and $\mathbb{E} \| u_{t+1} - u_t \| \to 0$. Since for any $t \geq 1$ it holds that
\begin{equation}\label{eq33}
\| z_{t+1} - z_t \| \leq \| A \| \| x_{t+1} - x_t \| + \sfrac{1}{\sigma\beta} \| u_{t+1} - u_t \| + \sfrac{1}{\sigma\beta} \| u_t - u_{t-1} \|.
\end{equation}
Applying the full expectation operator to \eqref{eq33}, we have $\mathbb{E} \| z_{t+1} - z_t \| \to 0$. Therefore, we obtain the following conclusion
\begin{equation}\nonumber
\lim_{t \to \infty} \mathbb{E} \| x_{t+1} - x_t \| = \lim_{t \to \infty} \mathbb{E} \| u_{t+1} - u_t \| = \lim_{t \to \infty} \mathbb{E} \| z_{t+1} - z_t \| = 0.
\end{equation}
This completes the proof.
\end{proof}

Before proving Theorem \ref{convergence}, we first consider the upper estimates of subgradients of the function $\Phi$ at $(x_t, z_t, u_t, x_{t-1}, u_{t-1})$ for every $t \geq 1$ in the following lemma, where $\Phi$ is given in \eqref{not}.
\begin{lemma}\label{subgrad-bound}
Suppose that  $2\tau \geq \beta \|A\|^2$ and Assumption \ref{assum} is  satisfied. Let $\{ (x_t, z_t, u_t)_{t\geq 0}\}$ be a sequence generated by Algorithm \ref{alg:SADMM} and denote $(\partial\Phi_x^t, \partial\Phi_z^t, \partial\Phi_u^t, \partial\Phi_{x^{\prime}}^t, \partial\Phi_{u^{\prime}}^t) \in \partial \Phi(X_t)$. Then there exists a positive constant $p$ such that
\begin{equation}\nonumber
\mathbb{E}_{t-1} \| (\partial\Phi_x^t, \partial\Phi_z^t, \partial\Phi_u^t, \partial\Phi_{x^{\prime}}^t, \partial\Phi_{u^{\prime}}^t) \| \leq p(\mathbb{E}_{t-1} \| u_t - u_{t-1} \| + \mathbb{E}_{t-1} \| x_t - x_{t-1} \| + \| x_{t-1} - x_{t-2} \| ) + \Gamma_{t-1}
\end{equation}
for any $t \geq 1$. 
\end{lemma}

\begin{proof}
From the definition of $\Phi$, for any $t \geq 1$, we have
\begin{subequations}\label{eqe18}
\begin{align}
&\partial\Phi_{x}^t = \nabla H(x_t) + A^{T} u_t + \beta A^{T} (Ax_t - z_t) + 2C_1(x_t - x_{t-1}) + 2 \sigma C_0 B^{T} (A^{T} (u_t - u_{t-1}) + \sigma B(x_t - x_{t-1})),\label{eq36}\\
&\partial\Phi_{z}^t = u_{t-1} - u_t + \beta A(x_{t-1} - x_t),\label{eq37}\\
&\partial\Phi_{u}^t = Ax_t - z_t + 2C_0 A(A^T (u_t - u_{t-1}) + \sigma B(x_t - x_{t-1})),\label{eq38}\\
&\partial\Phi_{x^{\prime}}^t = -2C_0\sigma B^T(A^T (u_t - u_{t-1}) + \sigma B(x_t - x_{t-1})) - 2C_1 (x_t - x_{t-1}),\label{eq39}\\
&\partial\Phi_{u^{\prime}}^t = -2C_0A(A^T (u_t - u_{t-1}) + \sigma B(x_t - x_{t-1})).\label{eq40}
\end{align}
\end{subequations}
By \eqref{sub2} and \eqref{sub3} we have
\begin{equation}\nonumber
\begin{aligned}
A^T u_{t+1} &= A^T u_t + \sigma \beta A^T (Ax_{t+1} - z_{t+1})\\
&=A^T u_t + (\sigma - 1) \beta A^T (Ax_{t+1} - z_{t+1}) + \beta A^T (Ax_{t+1} - z_{t+1})\\
& = (1 - \sfrac{1}{\sigma}) A^T (u_{t+1} - u_t) + A^T u_t  + \beta A^T (Ax_{t} - z_{t+1}) + \beta A^T A(x_{t+1} - x_{t})\\
& = (1 - \sfrac{1}{\sigma}) A^T (u_{t+1} - u_t) + \tau (x_t - x_{t+1}) - \tnabla  H(x_t) + \beta A^T A(x_{t+1} - x_{t})\\
& = (1 - \sfrac{1}{\sigma}) A^T (u_{t+1} - u_t) + B(x_t - x_{t+1}) - \tnabla  H(x_t). \\
\end{aligned}
\end{equation}
Then it follows from \eqref{eq36} that we have
\begin{equation}\nonumber
\begin{aligned}
\partial\Phi_{x}^t &= \nabla H(x_t) - \tnabla  H(x_{t-1}) + (1 - \sfrac{1}{\sigma}) A^T (u_t - u_{t-1}) + B(x_{t-1} - x_t) + \sfrac{1}{\sigma} A^T(u_t - u_{t-1})\\
& \qquad + 2 C_1 (x_t - x_{t-1}) + 2\sigma C_0 B^T(A^T(u_t - u_{t-1}) + \sigma B(x_t - x_{t-1}))\\
&= \nabla H(x_t) - \tnabla  H(x_{t-1}) + \Pa{ (1 + 2\sigma C_0 B^T)A^T} (u_t - u_{t-1}) + ( 2\sigma^2 C_0 B^TB - B + 2C_1)  (x_t - x_{t-1}).
\end{aligned}
\end{equation}
Hence, we get 
\begin{equation}\nonumber
\begin{aligned}
\| \partial\Phi_{x}^t  \| &\leq \| \nabla H(x_t) - \tnabla  H(x_{t-1})\| + (2C_1 +  2\sigma^2 C_0 \tau^2 + \tau) \| x_t - x_{t-1} \| + \left( 1 + 2\sigma\tau C_0 \right) \| A\| \| u_t - u_{t-1} \|\\
&\leq \| \nabla H(x_t) - \nabla H(x_{t-1})\| + \| \nabla H(x_{t-1}) - \tnabla  H(x_{t-1})\| \\
&\qquad+ (2C_1 +  2\sigma^2 C_0 \tau^2 + \tau) \| x_t - x_{t-1} \| + \left( 1 + 2\sigma\tau C_0 \right) \| A\| \| u_t - u_{t-1} \|\\
&\leq \| \nabla H(x_{t-1}) - \tnabla  H(x_{t-1})\| +  (L+ 2C_1 +  2\sigma^2 C_0 \tau^2 + \tau) \| x_t - x_{t-1} \| + \left( 1 + 2\sigma\tau C_0 \right) \| A\| \| u_t - u_{t-1} \|.\\
\end{aligned}
\end{equation}
From \eqref{eq37}, \eqref{eq38}, \eqref{eq39} and \eqref{eq40}, we easily get the following estimate :
$$
\| \partial\Phi_{z}^t \| \leq \| u_t - u_{t-1} \| + \beta \|A\| \| x_t - x_{t-1} \|,
$$
\begin{equation}\nonumber
\begin{aligned}
\| \partial\Phi_{u}^t \| &\leq \|Ax_t - z_t \| + 2C_0 \| A\|^2 \| u_t -u_{t-1} \| + 2\sigma \tau C_0 \| A\| \| x_t - x_{t-1} \|\\
&\leq (\sfrac{1}{\sigma\beta} +2C_0 \| A\| ^2 )\| u_t -u_{t-1} \| + 2\sigma \tau C_0 \| A\| \| x_t - x_{t-1} \|,
\end{aligned}
\end{equation}
\begin{equation}\nonumber
\| \partial\Phi_{x^{\prime}}^t \| \leq  (2\sigma^2C_0\tau^2 + 2C_1) \| x_t - x_{t-1} \| + 2\sigma C_0 \tau \|A\| \| u_t - u_{t-1}\|,
\end{equation}
\begin{equation}\nonumber
\| \partial\Phi_{u^{\prime}}^t \| \leq 2C_0 \| A\|^2 \| u_t - u_{t-1} \| + 2C_0 \sigma \tau \|A\| \| x_t - x_{t-1} \|.
\end{equation}
Therefore, we have
\begin{equation}\label{eq48}
\begin{aligned}
&\mathbb{E}_{t-1} \| (\partial\Phi_{x}^t, \partial\Phi_{z}^t, \partial\Phi_{u}^t, \partial\Phi_{x^{\prime}}^t, \partial\Phi_{u^{\prime}}^t) \| \\
&\leq \mathbb{E}_{t-1} [\| \partial\Phi_{x}^t \| + \|\partial\Phi_{z}^t \| + \| \partial\Phi_{u}^t \| + \| \partial\Phi_{x^{\prime}}^t \| + \| \partial\Phi_{u^{\prime}}^t) \| ]\\
&\leq p \left( \mathbb{E}_{t-1} \| u_t -u_{t-1} \| + \mathbb{E}_{t-1} \| x_t - x_{t-1} \| \right) + \mathbb{E}_{t-1} \| \nabla H(x_{t-1}) - \tnabla  H(x_{t-1}) \| \\
&\leq p \left( \mathbb{E}_{t-1} \| u_t -u_{t-1} \| + \mathbb{E}_{t-1} \| x_t - x_{t-1} \| + \| x_{t-1} - x_{t-2} \| \right) + \Gamma_{t-1},
\end{aligned}
\end{equation}
where $p = \max\{ L + 4C_1 + 4\sigma\tau C_0 (\sigma\tau + \| A \|) + \tau + \beta\|A \| + V_2,~1 + \sfrac{1}{\sigma\beta} + 4C_0\| A\| (\sigma\tau + \| A\| ) + (\sfrac{2}{\sigma} - 1) \| A\| \}.$ 
\end{proof}

We define $\Omega \eqdef \Omega (\{ X_t \}_{t \geq 1})$ is the set of cluster points of $\{ X_t \}_{t \geq 1}$, which is nonempty due to the boundedness of $\{ X_t \}_{t \geq 1}.$ The following lemma will introduce some properties of $\Omega$.

\begin{lemma}\label{lemma3}
The following statements are true:
\begin{itemize}
\item[(i)] $\sum_{t=1}^{\infty} \| X_t - X_{t-1} \|^2 \leq \infty~~a.s. $ and $\| X_t - X_{t-1} \| \to 0~~a.s.$;

\item[(ii)] $\mathbb{E} \Phi(X_t) \to \Phi^{*}$, where $\Phi^{*} \in [\underline{\Phi}, \infty)$;

\item[(iii)] $\mathbb{E} \dist(0, \partial \Phi(X_t)) \to 0$;

\item[(iv)] The set $\Omega$ is non-empty, and for all $X^* \in \Omega$, $\mathbb{E}\dist(0, \partial \Phi(X^*)) = 0$;

\item[(v)] $\dist(X_t, \Omega) \to 0~a.s.$;

\item[(vi)] $\Omega$ is $a.s.$ compact and connected;

\item[(vii)] $\mathbb{E} \Phi(X^*) = \Phi^*$ for all $X^* \in \Omega$.
\end{itemize}
\end{lemma}

\begin{proof}
${\bf (i)}$ From \eqref{eq31}, we know
\begin{equation}\nonumber
\mathbb{E}_t [ \Psi_{t+1} + \tilde{\eta}  \| x_{t+1} - x_t \|^2 + \sfrac{1}{\sigma\beta} \| u_{t+1} - u_t \|^2  + C_2 \| x_t - x_{t-1} \|^2 ] \leq \Psi_t.
\end{equation}
Adding $\tilde{\eta}  \| x_{t} - x_{t-1} \|^2 + \sfrac{1}{\sigma\beta} \| u_{t} - u_{t-1} \|^2 $ on the both sides, we have
\begin{equation}\nonumber
\begin{aligned}
&\mathbb{E}_t [ \Psi_{t+1} + \tilde{\eta}  \| x_{t+1} - x_t \|^2 + \sfrac{1}{\sigma\beta} \| u_{t+1} - u_t \|^2 ] + (\tilde{\eta} + C_2)  \| x_{t} - x_{t-1} \|^2 + \sfrac{1}{\sigma\beta} \| u_{t} - u_{t-1} \|^2\\
& \leq \Psi_t + \tilde{\eta}  \| x_{t} - x_{t-1} \|^2 + \sfrac{1}{\sigma\beta} \| u_{t} - u_{t-1} \|^2.
\end{aligned}
\end{equation}
The supermartingale convergence theorem implies that $\sum_{t=1}^{t=\infty} \| u_t - u_{t-1} \|^2 < \infty~a.s.$ and $\sum_{t=1}^{t=\infty} \| x_t - x_{t-1} \|^2 < \infty~a.s.$. Furthermore, from $\eqref{eq33}$, we have $\sum_{t=1}^{t=\infty} \| z_t - z_{t-1} \|^2 < \infty~a.s.$ and it follows that 
\begin{equation}\nonumber
\lim_{t \to \infty} \| x_t - x_{t-1} \| = \lim_{t \to \infty} \| u_t - u_{t-1} \| = \lim_{t \to \infty} \| z_t - z_{t-1} \| = 0,~~a.s.
\end{equation}
This proves Claim $(i)$.\\

${\bf (ii)}$ The supermartingale convergence theorem also ensures $ \Psi_t + \tilde{\eta}  \| x_{t} - x_{t-1} \|^2 + \sfrac{1}{\sigma\beta} \| u_{t} - u_{t-1} \|^2$ convergence $a.s.$ to a finite, positve random variable. Because $\lim_{t \to \infty} \| X_t - X_{t-1} \| = 0~a.s.$ and $\mathbb{E} \Upsilon_t \to 0$, we can say $\lim_{t \to \infty} \mathbb{E}  \Psi_t + \tilde{\eta}  \| x_{t} - x_{t-1} \|^2 + \sfrac{1}{\sigma\beta} \| u_{t} - u_{t-1} \|^2 = \lim_{t \to \infty} \mathbb{E}\Phi(X_t) \in [\underline{\Phi}, \infty)$, this is implies Claim $(ii)$.\\

${\bf (iii)}$ Claim $(iii)$ holds because, by \eqref{eq48},
\begin{equation}\nonumber
\begin{aligned}
&\mathbb{E}_{t-1} \| (\partial\Phi_{x}^t, \partial\Phi_{z}^t, \partial\Phi_{u}^t, \partial\Phi_{x^{\prime}}^t, \partial\Phi_{u^{\prime}}^t) \| \\
&\leq p \left( \mathbb{E}_{t-1} \| u_t -u_{t-1} \| + \mathbb{E}_{t-1} \| x_t - x_{t-1} \| + \| x_{t-1} - x_{t-2} \| \right) + \Gamma_{t-1}.
\end{aligned}
\end{equation}
We have that $\| u_t -u_{t-1} \| \to 0~a.s.$, $\| x_t -x_{t-1} \| \to 0~a.s.$ and $\mathbb{E}\Gamma_{t-1} \to 0.$ This ensures that $\mathbb{E} \| \partial\Phi(X^t)\| \to 0$.\\

${\bf (iv)}$ To prove Claim $(iv)$, suppose $X_{*} = (x^*, y^*, z^*, u^*,x^*, u^*)$ is a limit point of the sequence $\{ X_t \}_{t=0}^{\infty}$. This means there exists a subsequence $X_{t_k}$ satisfying $\lim_{k \to \infty} X_{t_k} \to X_{*}$. From \eqref{sub1} we have, for any $k \geq 1$,
\begin{equation}\nonumber
\begin{aligned}
&F(z_{t_k}) + \langle u_{t_{k - 1}}, Ax_{t_{k-1}} - z_{t_{k-1}} \rangle + \sfrac{\beta}{2} \| Ax_{t_{k-1}} - z_{t_{k-1}} \|^2\\
&\leq F(z^*) + \langle u_{t_{k - 1}}, Ax_{t_{k-1}} - z_{*} \rangle + \sfrac{\beta}{2} \| Ax_{t_{k-1}} - z_{*} \|^2.
\end{aligned}
\end{equation}
Because the $\| u_t - u_{t-1} \| \to 0~a.s.$, we have $\| Ax_{t} - z_{t-1}\| \to 0~a.s.$ and $\| Ax_{t} - z_{*}\| \to 0~a.s.$. Taking the limit superior as $k \to \infty$ on the both sides of the above inequalities, we get
\begin{equation}\nonumber
\lim\sup_{k \to \infty} F(z_{t_k}) \leq F(z^*),
\end{equation}
which combined with the lower semicontinuity of $F$, lead to 
\begin{equation}\nonumber
\lim_{k \to \infty} F(z_{t_k}) = F(z^*).
\end{equation}
Because the function $H$ is continuous, it follows that
\begin{equation}\nonumber
\lim_{k \to \infty} \Phi(x_{t_k}, z_{t_k}, u_{t_k}, x_{t_{k-1}}, u_{t_{k-1}}) = \Phi(x^*, z^*, u^*, x^*, u^*) = \Phi(X^*).
\end{equation}
Claim $(iii)$ ensures that $X^*$ is a critical point of $\Phi$ because $\mathbb{E}\dist(0, \partial \Phi(X_t)) \to 0$ as $t \to \infty$ and $\partial \Phi(X^*)$ is closed.\\

${\bf (v)~and~(vi)}$~Claim $(v)$ and $(vi)$ hold for any sequence satisfying $\| X_t - X_{t-1}\| \to 0~a.s.$.\\

${\bf (vii)}$~Finally, we must show that $\Phi$ has constant expectation over $\Omega$. From Claim $(ii)$, we have that $\mathbb{E}\Phi(X_t) \to \Phi^*$, which implies that $\mathbb{E}\Phi(X_{t_k}) \to \Phi(X^*)$, so $\mathbb{E} \Phi(X^*) = \Phi^*$ for all $X^* \in \Omega$.
\end{proof}

Next, we present  the following random version of the KL inequality, whose proof  is same as that of  Lemma C.4 in \cite{DTLDS}, hence here we omit its proof.
\begin{theorem}\label{KL}
Let $\{ X_t \}_{t=0}^{\infty}$ be a bounded sequence of iterates of Algorithm \ref{alg:SADMM} using a variance-reduced gradient estimator, and suppose that $X_t$ is not a critical point after a finite number of iterations. Let $\Phi$ be a semialgebraic function satisfying the Kurdyka-$\L$ojasiewicz property (see Definition \ref{KL1}) with exponent $\theta$. Then there exists an index $m$ and desingularizing function $\phi = a r^{1-\theta}$ so that the following holds almost surely:
\begin{equation}\nonumber
\phi^{\prime} (\mathbb{E}[\Phi(X_t) - \Phi_{t}^{*}]) \mathbb{E}\dist(0, \partial \Phi(X_t)) \geq 1,~~~~~\forall t > m,
\end{equation}
where $\Phi_{t}^{*}$ is an non-decreasing sequence converging to $\mathbb{E}\Phi(X^*)$ for some $X^* \in \partial \Omega$, where $\Omega$ is the set of cluster points of $\{ X_t \}_{t \geq 1}$.
\end{theorem}

\begin{proof}[Proof of Theorem \ref{convergence}]
If $\theta \in (0, \sfrac{1}{2})$, then $\Phi$ satisfies the KL property with exponent $\sfrac{1}{2}$, so we consider only the case $\theta \in [\sfrac{1}{2}, 1).$ By Theorem \ref{KL} there exists a function $\phi_0(r) = ar^{1-\theta}$ such that, almost surely,
\begin{equation}\nonumber
\phi_0^{\prime} (\mathbb{E}[\Phi(X_t) - \Phi_t^*]) \mathbb{E}\dist(0, \partial\Phi(X_t)) \geq 1, \forall t >m.
\end{equation}
Lemma \ref{subgrad-bound} provides a bound on $\mathbb{E}\dist(0, \partial\Phi(X_t))$,
\begin{equation}\nonumber
\begin{aligned}
\mathbb{E}\dist(0, \partial\Phi(X_t)) &\leq \mathbb{E}\| \partial \Phi(X_t) \| \\
&\leq p(\mathbb{E}\| u_t -u_{t-1}\| + \mathbb{E} \| x_t - x_{t-1}\| + \mathbb{E} \| x_{t-1} - x_{t-2} \|) + \mathbb{E} \Gamma_{t-1}\\
&\leq p ( \sqrt{\mathbb{E}\| u_t -u_{t-1}\|^2} + \sqrt{\mathbb{E} \| x_t - x_{t-1}\|^2} + \sqrt{\mathbb{E} \| x_{t-1} - x_{t-2}\|^2}) + \sqrt{s \mathbb{E} \Upsilon_{t-1}}.
\end{aligned}
\end{equation}
The final inequality is Jense$n^{\prime}$s. Because of $\Gamma_t = \sum_{i=1}^{s} \| v_t^i \|$ for some vectors $v_k^I$, we have $\mathbb{E} \Gamma_t = \mathbb{E} \sum_{i=1}^s \| v_k^i \| \leq \mathbb{E} \sqrt{s\sum_{i=1}^s \| v_k^i \|^2} \leq \sqrt{s\mathbb{E}\Upsilon_t}$. We can bound the term $\sqrt{\mathbb{E}\Upsilon_t}$:
\begin{equation}\label{eq64}
\begin{aligned}
\sqrt{\mathbb{E}\Upsilon_t} &\leq \sqrt{(1 - \rho) \mathbb{E}\Upsilon_{t-1} + V_r \mathbb{E} \| x_t - x_{t-1} \|^2 + V_r \mathbb{E} \| x_{t-1} - x_{t-2} \|^2}\\
&\leq \sqrt{1-\rho} \sqrt{\mathbb{E}\Upsilon_{t-1}} + \sqrt{V_r} \sqrt{\mathbb{E} \| x_t -x_{t-1} \|^2} + \sqrt{V_r} \sqrt{\mathbb{E} \| x_{t-1} -x_{t-2} \|^2}\\
&\leq (1 - \sfrac{\rho}{2}) \sqrt{\mathbb{E}\Upsilon_{t-1}} + \sqrt{V_r} \sqrt{\mathbb{E} \| x_t -x_{t-1} \|^2} + \sqrt{V_r} \sqrt{\mathbb{E} \| x_{t-1} -x_{t-2} \|^2}.
\end{aligned}
\end{equation}
This implies that
\begin{equation}\nonumber
\sqrt{\mathbb{E}\Upsilon_{t-1}} \leq \sfrac{2}{\rho} (\sqrt{\mathbb{E}\Upsilon_{t-1}} - \sqrt{\mathbb{E}\Upsilon_t}) + \sfrac{2}{\rho} \sqrt{V_r} \sqrt{\mathbb{E} \| x_t - x_{t-1} \|^2} + \sfrac{2}{\rho} \sqrt{V_r} \sqrt{\mathbb{E} \| x_{t-1} - x_{t-2} \|^2},
\end{equation}
and multiply $\sqrt{s}$,
\begin{equation}\label{eq67}
\sqrt{s\mathbb{E}\Upsilon_{t-1}} \leq \sfrac{2\sqrt{s}}{\rho} (\sqrt{\mathbb{E}\Upsilon_{t-1}} - \sqrt{\mathbb{E}\Upsilon_t}) + \sfrac{2\sqrt{s}\sqrt{V_r}}{\rho} \sqrt{\mathbb{E}\| x_t - x_{t-1} \|^2} + + \sfrac{2\sqrt{s}\sqrt{V_r}}{\rho} \sqrt{\mathbb{E}\| x_{t-1} - x_{t-2} \|^2}.
\end{equation}
Then we have
\begin{equation}\nonumber
\begin{aligned}
&\mathbb{E}\dist(0, \partial \Phi(X_t))\\
&\leq p \sqrt{\mathbb{E}\| u_t -u_{t-1} \|^2} + (p + \sfrac{2\sqrt{s} \sqrt{V_r}}{\rho}) \sqrt{\mathbb{E} \| x_t - x_{t-1} \|^2} + (p + \sfrac{2\sqrt{s} \sqrt{V_r}}{\rho}) \sqrt{\mathbb{E} \| x_{t-1} - x_{t-2} \|^2}\\
&~~~~ + \sfrac{2\sqrt{s}}{\rho} (\sqrt{\mathbb{E}\Upsilon_{t-1}} - \sqrt{\mathbb{E}\Upsilon_t})\\
&\leq K_1 \sqrt{\mathbb{E} \| u_t - u_{t-1} \|^2} + K_1 \sqrt{\mathbb{E} \| x_t - x_{t-1} \|^2} + K_1 \sqrt{\mathbb{E} \| x_{t-1} - x_{t-2} \|^2} + \sfrac{2\sqrt{s}}{\rho} (\sqrt{\mathbb{E}\Upsilon_{t-1}} - \sqrt{\mathbb{E}\Upsilon_t}),\\
\end{aligned}
\end{equation}
where $K_1 \eqdef p + \sfrac{2\sqrt{s} \sqrt{V_r}}{\rho}$. Denote ${\bf C_t}$ as the right side of this inequality:
\begin{equation}\nonumber
{\bf C_t} \eqdef K_1 \sqrt{\mathbb{E} \| u_t - u_{t-1} \|^2} + K_1 \sqrt{\mathbb{E} \| x_t - x_{t-1} \|^2} + K_1 \sqrt{\mathbb{E} \| x_{t-1} - x_{t-2} \|^2} + \sfrac{2\sqrt{s}}{\rho} (\sqrt{\mathbb{E}\Upsilon_{t-1}} - \sqrt{\mathbb{E}\Upsilon_t}).
\end{equation}
We then have
$$
\phi_0^{\prime} (\mathbb{E} [\Phi(X_t) - \Phi(X_t^*)]) {\bf C_t} \geq 1, ~~~~~\forall t > m.
$$
By the definition of $\phi_0$, this is equivalent to 
\begin{equation}\label{eq71}
\sfrac{a(1- \theta){\bf C_t}}{(\mathbb{E}[\Phi(X_t) - \Phi_t^*])^{\theta}} \geq 1,~~~~~\forall t >m.
\end{equation}
We would like the inequality above to hold for $\Psi_t$ rather than $\Phi(X_t)$. Repalcing $\mathbb{E} \Phi(X_t)$ with $\mathbb{E} \Psi_t$ introduces a term of $\mathcal{O} \Pa{  (\mathbb{E} [ \| \tnabla  H(x_{t-1}) - \nabla H(x_{t-1})\|^2 + \Upsilon_t + \| x_t - x_{t-1} \|^2 ])^\theta }$ in the denominator. We show that inequality \eqref{eq71} still holds after this adjustment because these terms are small compared to~${\bf C_t}$.\\
\\
The quantity ${\bf C_t} \geq \mathcal{O} (\sqrt{\mathbb{E}\| x_t - x_{t-1} \|^2} + \sqrt{\mathbb{E}\| x_{t-1} - x_{t-2} \|^2} + \sqrt{\mathbb{E} \| u_t - u_{t-1} \|^2} + \sqrt{\mathbb{E} \Upsilon_{t-1}})$, and because $\mathbb{E}\| \tnabla  H(x_{t-1}) - \nabla H(x_{t-1})\|^2,~\mathbb{E}\Upsilon_t \to  0$,~$\mathbb{E}\| x_t - x_{t-1}\| \to 0$ and $\theta > \sfrac{1}{2}$, there exists an index $m$ and a constant $c> 0$ such that
\begin{equation}\nonumber
\begin{aligned}
&  \mathbb{E} \big[ \sfrac{16}{\sigma \beta \lambda_{m}}\| \tnabla  H(x_{t-1}) - \nabla H(x_{t-1})\|^2 +  (\sfrac{32}{\sigma \beta \lambda_{m}} + \sfrac{\eta}{2}) \sfrac{\Upsilon_t}{\rho} + \big( (\sfrac{32}{\sigma\beta\lambda_m} + \sfrac{\eta}{2}) (V_1 + \sfrac{V_\Upsilon}{\rho}) + C_2\big)\| x_t - x_{t-1}\|^2 \big] \\
&\leq  \mathbb{E} \big[ \sfrac{16}{\sigma \beta \lambda_{m}}\Upsilon_{t-1} + \big( \sfrac{16V_1}{\sigma \beta \lambda_{m}} + (\sfrac{32}{\sigma\beta\lambda_m} + \sfrac{\eta}{2}) (V_1 + \sfrac{V_\Upsilon}{\rho}) + C_2 \big) \| x_t - x_{t-1} \|^2 + \sfrac{16V_1}{\sigma \beta \lambda_{m}} \| x_{t-1} - x_{t-2} \|^2 \\
&~~~~+  (\sfrac{32}{\sigma \beta \lambda_{m}} + \sfrac{\eta}{2}) \sfrac{\Upsilon_t}{\rho} \big] \\
&\leq  \mathbb{E} \big[ (\sfrac{16}{\sigma \beta \lambda_{m}} + (\sfrac{32}{\sigma \beta \lambda_{m}} + \sfrac{\eta}{2}) \sfrac{1-\rho}{\rho} ) \Upsilon_{t-1} + \big( \sfrac{16V_1}{\sigma \beta \lambda_{m}} + (\sfrac{32}{\sigma\beta\lambda_m} + \sfrac{\eta}{2}) (V_1 + \sfrac{2V_\Upsilon}{\rho}) + C_2 \big) \| x_t - x_{t-1} \|^2 \\
&~~~~+ \big(  \sfrac{16V_1}{\sigma \beta \lambda_{m}} + (\sfrac{32}{\sigma\beta\lambda_m} + \sfrac{\eta}{2}) \sfrac{V_\Upsilon}{\rho} \big) \| x_{t-1} - x_{t-2} \|^2 \big] \\
&\leq \mathcal{O}  \big( \mathbb{E} [\Upsilon_{t-1} + \| x_t - x_{t-1} \|^2 + \| x_{t-1} - x_{t-2} \|^2] \big)\\
&\leq c {\bf C_t}^{\sfrac{1}{\theta}},~~~\forall t > m.
\end{aligned}
\end{equation}
Denote 
\begin{equation}\nonumber
\begin{aligned}
\Lambda_t \eqdef \sfrac{16}{\sigma \beta \lambda_{m}}\| \tnabla  H(x_{t-1}) - \nabla H(x_{t-1})\|^2 +  (\sfrac{32}{\sigma \beta \lambda_{m}} + \sfrac{\eta}{2}) \sfrac{\Upsilon_t}{\rho} + \Pa{ (\sfrac{32}{\sigma\beta\lambda_m} + \sfrac{\eta}{2}) (V_1 + \sfrac{V_\Upsilon}{\rho}) + C_2} \| x_t - x_{t-1}\|^2.
\end{aligned}
\end{equation}
Because the terms above are small compared to ${\bf C_t}$, there exists a constant $+\infty > d > c$ such that
\begin{equation}\nonumber
\begin{aligned}
\sfrac{ad(1-\theta) {\bf C_t}}{(\mathbb{E}[\Phi(X_t) - \Phi_t^*])^{\theta} + \left(  \mathbb{E} \Lambda_t \right)^{\theta}} \geq 1
\end{aligned}
\end{equation}
for all $t > m$. Using the fact that $(a + b)^{\theta} \leq a^{\theta} + b^{\theta}$ because $\theta \in [\sfrac{1}{2}, 1)$, we have
\begin{equation}\nonumber
\begin{aligned}
\sfrac{ad(1 - \theta){\bf C_t}}{(\mathbb{E}[\Psi_t - \Phi_t^*])^{\theta}} 
= \sfrac{ad(1 - \theta){\bf C_t}}{(\mathbb{E}[\Phi_{t} - \Phi_t^* + \Lambda_t ])^{\theta}}
\geq \sfrac{ad(1 - \theta){\bf C_t}}{(\mathbb{E}[\Phi_{t} - \Phi_t^*])^\theta + (\mathbb{E} \Lambda_t)^{\theta}} 
\geq 1,~~\forall t >m.
\end{aligned}
\end{equation}
Therefore, with $\phi(r) = adr^{1-\theta}$,
\begin{equation}\nonumber
\phi^{\prime}(\mathbb{E}[ \Psi_t - \Phi_t^*] ) {\bf C_t} \geq 1, ~~~\forall t >m.
\end{equation}
By the concavity of $\phi$,
\begin{equation}\nonumber
\begin{aligned}
\phi(\mathbb{E} [\Psi_t - \Phi_t^*]) - \phi(\mathbb{E}[\Psi_{t+1} - \Phi_{t+1}^*])& \geq \phi^{\prime} (\mathbb{E}[\Psi_t - \Phi_t^*]) (\mathbb{E} [\Psi_t - \Phi_t^* + \Phi_{t+1}^* - \Psi_{t+1}])\\
&\geq \phi^{\prime}(\mathbb{E} [\Psi_t - \Phi_t^*]) (\mathbb{E}[\Psi_t - \Psi_{t+1}]),
\end{aligned}
\end{equation}
where the last inequality follows from the fact that $\Phi_t^*$ is non-decreasing. Denote $\Delta_{p,q} \eqdef \phi(\mathbb{E}[\Psi_p - \Phi_p^*]) - \phi (\mathbb{E}[\Psi_q - \Phi_q^*]),$ we have shown
\begin{equation}\nonumber
\Delta_{t,t+1}{\bf C_t} \geq \mathbb{E}[\Psi_t - \Psi_{t+1}].
\end{equation}
Using \eqref{eq31}, we can bound $\mathbb{E}[\Psi_t - \Psi_{t+1}]$ below by both $\mathbb{E}\| x_{t+1} - x_t \|^2$,  $\mathbb{E}\| x_{t} - x_{t -1}\|^2$ and $\mathbb{E}\| u_{t+1} - u_t \|^2$. Specifically,
\begin{equation}\label{eq78}
\begin{aligned}
\Delta_{t,t+1} {\bf C_t}&\geq \tilde{\eta} \mathbb{E} \| x_{t+1}  - x_t \|^2 + C_2 \mathbb{E} \| x_{t}  - x_{t-1} \|^2 + \sfrac{1}{\sigma\beta}\mathbb{E} \| u_{t+1} - u_t \|^2 \\
&\geq K \mathbb{E} \| x_{t+1} - x_t \|^2 + K \mathbb{E} \| x_{t} - x_{t-1} \|^2+ K \mathbb{E} \| u_{t+1} - u_t \|^2,
\end{aligned}
\end{equation}
where $K = \min\{\tilde{\eta}, \sfrac{1}{\sigma\beta}, C_2\}$. Applying Young's inequality to \eqref{eq78} yields
\begin{equation}\nonumber
\begin{aligned}
&2\sqrt{\mathbb{E}\| x_{t+1} - x_t \|^2 + \mathbb{E}\| x_{t} - x_{t-1} \|^2 + \mathbb{E} \| u_{t+1} - u_t \|^2}\\
&\leq 2\sqrt{K^{-1}{\bf C_t} \Delta_{t,t+1}} \leq \sfrac{{\bf C_t}}{2K_1} + \sfrac{2K_1 \Delta_{t,t+1}}{K}\\
&\leq \sfrac{\sqrt{\mathbb{E}\| u_t - u_{t-1}\|^2}}{2} + \sfrac{\sqrt{\mathbb{E}\| x_t - x_{t-1} \|^2}}{2} + \sfrac{\sqrt{\mathbb{E}\| x_{t-1} - x_{t-2} \|^2}}{2} + \sfrac{\sqrt{s}}{K_1\rho} (\sqrt{\mathbb{E}\Upsilon_{t-1}} - \sqrt{\mathbb{E}\Upsilon_t}) + \sfrac{2K_1 \Delta_{t,t+1}}{K}.\\
\end{aligned}
\end{equation}
Then we have
\begin{equation}\label{eq80}
\begin{aligned}
&\sqrt{\mathbb{E} \| x_{t+1} - x_t \|^2} + \sqrt{\mathbb{E}\| x_{t} - x_{t-1} \|^2} + \sqrt{\mathbb{E} \| u_{t+1} - u_t \|^2}\\
&\leq \sqrt{2} \sqrt{\mathbb{E}\| x_{t+1} - x_t \|^2 + \mathbb{E}\| x_{t} - x_{t-1} \|^2 + \mathbb{E}\| u_{t+1} - u_t \|^2}\\
&\leq  \sfrac{\sqrt{2}\sqrt{\mathbb{E}\| u_t - u_{t-1}\|^2}}{4} + \sfrac{\sqrt{2}\sqrt{\mathbb{E}\| x_t - x_{t-1} \|^2}}{4} + \sfrac{\sqrt{2}\sqrt{\mathbb{E}\| x_{t-1} - x_{t-2} \|^2}}{4} + \sfrac{\sqrt{2}\sqrt{s}}{2K_1\rho} (\sqrt{\mathbb{E}\Upsilon_{t-1}} - \sqrt{\mathbb{E}\Upsilon_t}) + \sfrac{\sqrt{2}K_1 \Delta_{t,t+1}}{K}.\\
\end{aligned}
\end{equation}
Summing inequality \eqref{eq80} from $t=m$ to $t=i$,
\begin{equation}\nonumber
\begin{aligned}
&\sum_{t=m}^{i} \sqrt{\mathbb{E} \| x_{t+1} - x_t \|^2} + \sum_{t=m}^{i} \sqrt{\mathbb{E} \| x_{t} - x_{t-1} \|^2} + \sum_{t=m}^i \sqrt{\mathbb{E} \| u_{t+1} - u_t \|^2}\\
&\leq \sfrac{\sqrt{2}}{4} \sum_{t=m}^{i} \sqrt{\mathbb{E} \| x_t - x_{t-1} \|^2}  + \sfrac{\sqrt{2}}{4} \sum_{t=m}^{i} \sqrt{\mathbb{E} \| x_{t-1} - x_{t-2} \|^2} + \sfrac{\sqrt{2}}{4} \sum_{t=m}^i \sqrt{\mathbb{E} \| u_t - u_{t-1} \|^2} \\
&\qquad+ \sfrac{\sqrt{2s}}{2K_1\rho} \sum_{t=m}^i (\sqrt{\mathbb{E} \Upsilon_{t-1} } - \sqrt{\mathbb{E}\Upsilon_t}) + \sfrac{\sqrt{2} K_1}{K} \sum_{t=m}^{i} \Delta_{t, t+1}\\
&\leq \sfrac{\sqrt{2}}{4} \sum_{t=m}^{i} \sqrt{\mathbb{E} \| x_t - x_{t-1} \|^2} + \sfrac{\sqrt{2}}{4} \sum_{t=m}^{i} \sqrt{\mathbb{E} \| x_{t-1} - x_{t-2} \|^2} + \sfrac{\sqrt{2}}{4} \sum_{t=m}^i \sqrt{\mathbb{E} \| u_t - u_{t-1} \|^2} \\
&\qquad+ \sfrac{\sqrt{2s}}{2K_1\rho} (\sqrt{\mathbb{E} \Upsilon_{m-1} } - \sqrt{\mathbb{E}\Upsilon_i}) + \sfrac{\sqrt{2} K_1}{K} \Delta_{m, i+1}.\\
\end{aligned}
\end{equation}
This implies that
\begin{equation}\label{eq82}
\begin{aligned}
&\sum_{t=m}^{i} \sqrt{\mathbb{E} \| x_{t+1} - x_t \|^2} + \sum_{t=m}^{i} \sqrt{\mathbb{E} \| x_{t} - x_{t-1} \|^2} + \sum_{t=m}^i \sqrt{\mathbb{E} \| u_{t+1} - u_t \|^2}\\
&\leq \sfrac{2\sqrt{2} + 1}{6} \sqrt{\mathbb{E} \| x_m - x_{m-1} \|^2} + \sfrac{2\sqrt{2} + 1}{6} \sqrt{\mathbb{E} \| x_{m-1} - x_{m-2} \|^2} +  \sfrac{2\sqrt{2} + 1}{6} \sqrt{\mathbb{E} \| u_m - u_{m-1} \|^2}\\
&\qquad+  \sfrac{(2\sqrt{2}+1)\sqrt{s}}{3 K_1 \rho} \sqrt{\mathbb{E} \Upsilon_{m-1} }  + \sfrac{(4\sqrt{2} + 2)K_1}{3K} \Delta_{m, i+1}\\ 
&\leq \sqrt{\mathbb{E} \| x_m - x_{m-1} \|^2} + \sqrt{\mathbb{E} \| x_{m-1} - x_{m-2} \|^2}  + \sqrt{\mathbb{E} \| u_m - u_{m-1} \|^2} + \sfrac{(2\sqrt{2}+1)\sqrt{s}}{3 K_1 \rho} \sqrt{\mathbb{E} \Upsilon_{m-1} } + K_3 \Delta_{m, i+1},
\end{aligned}
\end{equation}
where $K_3 = \sfrac{(4\sqrt{2} + 2)K_1}{3K}$. Applying Jensen's inequality on the left leads to
\begin{equation}\nonumber
\begin{aligned}
&\sum_{t=m}^{i} \mathbb{E} \| x_{t+1} - x_t \| + \sum_{t=m}^{i} \mathbb{E} \| x_{t} - x_{t-1} \| + \sum_{t=m}^i  \mathbb{E} \| u_{t+1} - u_t \|\\
&\leq \sum_{t=m}^{i} \sqrt{\mathbb{E} \| x_{t+1} - x_t \|^2} + \sum_{t=m}^{i} \sqrt{\mathbb{E} \| x_{t} - x_{t-1} \|^2} + \sum_{t=m}^i \sqrt{\mathbb{E} \| u_{t+1} - u_t \|^2}\\
&\leq \sqrt{\mathbb{E} \| x_m - x_{m-1} \|^2} + \sqrt{\mathbb{E} \| x_{m-1} - x_{m-2} \|^2}  + \sqrt{\mathbb{E} \| u_m - u_{m-1} \|^2} + \sfrac{(2\sqrt{2}+1)\sqrt{s}}{3 K_1 \rho} \sqrt{\mathbb{E} \Upsilon_{m-1} } + K_3 \Delta_{m, i+1}.
\end{aligned}
\end{equation}
Since $\Delta_{m, i+1}$ is bounded, we get
\begin{equation}\nonumber
\sum_{t=m}^{\infty} \mathbb{E} \| x_{t+1} - x_t \| < \infty,~~~\sum_{t=m}^{\infty} \mathbb{E} \| u_{t+1} - u_t \| < \infty.
\end{equation}
This, together with \eqref{eq33}, gives
\begin{equation}\nonumber
\sum_{t=m}^{\infty} \mathbb{E} \| z_{t+1} - z_t \| < \infty.
\end{equation}
Thus the proof is completed.
\end{proof}

\end{document}